\documentclass[12pt]{scrarticle}
\usepackage[utf8]{inputenc}
\usepackage{authblk}

\usepackage{tikz}
\usepackage{tikz-cd}
\usepackage{amssymb, url}
\usepackage{amsthm}
\usepackage{amsmath}
\usepackage{amsfonts}
\usepackage{bbm} 
\usepackage{dsfont} 
\usepackage{enumitem}
\usepackage[colorlinks=true]{hyperref} 

\hypersetup{urlcolor=blue,linkcolor=black,citecolor=black,colorlinks=true}
\usepackage[symbol=$\uparrow~$,numberlinked=true]{footnotebackref}
\usepackage[all]{xy}
\usepackage{verbatim}
\usepackage{soul}

\usepackage{calrsfs}
\DeclareMathAlphabet{\pazocal}{OMS}{zplm}{m}{n} 

\newtheorem{Theorem}{Theorem}[section]
\newtheorem*{Theorem*}{Theorem}
\newtheorem{Lemma}[Theorem]{Lemma}
\newtheorem{Fact}{Fact}
\newtheorem{Corollary}[Theorem]{Corollary}
\newtheorem*{Corollary*}{Corollary}
\newtheorem{Proposition}[Theorem]{Proposition}
\newtheorem*{Proposition*}{Proposition}
\newtheorem{MainTheorem}{Theorem}

\theoremstyle{definition}

\newtheorem{Definition}[Theorem]{Definition}
\newtheorem{Proposition/Definition}[Theorem]{Proposition/Definition}

\newtheorem{Example}[Theorem]{Example}
\newtheorem*{Example*}{Example}

\theoremstyle{remark}
\newtheorem{Remark}[Theorem]{Remark}
\newtheorem*{Remark*}{Remark}

\newcommand{\CAlg}{\mathsf{Alg}}

\newcommand{\CGrp}{\mathsf{Grp}}

\newcommand{\coker}{\operatorname{coker}}

\newcommand{\tMot}{t\mathbf{Mot}}

\newcommand{\tMotrat}{t\mathbf{Mot}^{\operatorname{rat}}}

\newcommand{\Spec}{\operatorname{Spec}}

\newcommand{\tr}{\operatorname{tr}}

\newcommand{\Hom}{\operatorname{Hom}}

\newcommand{\Aut}{\operatorname{Aut}}
\newcommand{\id}{\operatorname{id}}

\newcommand{\Gal}{\operatorname{Gal}}
\newcommand{\Frob}{\operatorname{Frob}}

\newcommand{\Mat}{\mathsf{Mat}} 
\newcommand{\GL}{\operatorname{GL}}

\newcommand{\KI}{K_{\infty}}
\newcommand{\CI}{\mathbb{C}_{\infty}}

\newcommand{\mI}{\mathfrak{m}_{\infty}}

\newcommand{\bbone}{\mathds{1}}

\newcommand{\bC}{\mathbb{C}}

\newcommand{\bF}{\mathbb{F}}
\newcommand{\bG}{\mathbb{G}}

\newcommand{\bN}{\mathbb{N}}

\newcommand{\bP}{\mathbb{P}}
\newcommand{\bQ}{\mathbb{Q}}
\newcommand{\bR}{\mathbb{R}}

\newcommand{\bZ}{\mathbb{Z}}

\newcommand{\cF}{\mathcal{F}}
\newcommand{\cG}{\mathcal{G}}
\newcommand{\cH}{\mathcal{H}}

\newcommand{\cL}{\mathcal{L}}

\newcommand{\cO}{\mathcal{O}}

\newcommand{\cR}{\mathcal{R}}

\newcommand{\cZ}{\mathcal{Z}}

\newcommand{\fm}{\mathfrak{m}}

\DeclareMathOperator{\Ext}{Ext}

\newcommand{\abs}[1]{\left\lvert #1 \right\rvert}

\newcommand{\ls}[1]{(\!(#1)\!)}
\newcommand{\cs}[1]{\langle #1 \rangle} 

\newcommand{\trdeg}{\mathrm{trdeg}}

\newcommand{\TT}{\KI\cs{t}}

\newcommand{\hdt}[2]{\partial_t^{(#1)}\!\!\left(#2\right)}
\newcommand{\hdte}[1]{\partial_t^{(#1)}}  

\newcommand{\norm}[1]{\left\lVert #1 \right\rVert}

\DeclareMathOperator{\Li}{Li}

\usepackage[mathscr]{eucal} 
\newcommand{\eC}{\mathscr{C}}

\title{Carlitz twists: their motivic cohomology, regulators, zeta values and polylogarithms}

\author{Quentin Gazda\thanks{Centre de Mathématiques Laurent Schwartz (CMLS), \'Ecole Polytechnique, Cour Vaneau F-91120 Palaiseau}~ and Andreas Maurischat\thanks{RWTH Aachen University, D-52062 Aachen}}
\date{}

\numberwithin{equation}{section}

\begin{document}

\maketitle

\begin{abstract}{\textbf{Abstract:}}
The integral $t$-motivic cohomology and the class module of a (rigid analytically trivial) Anderson $t$-motive were introduced by the first author in \cite{gazda2}. This paper is devoted to their determination in the particular case of tensor powers of the Carlitz $t$-motive, namely, the function field counterpart $\underline{A}(n)$ of Tate twists $\mathbb{Z}(n)$. We find out that these modules are in relation with fundamental objects of function field arithmetic: integral $t$-motivic cohomology governs linear relations among Carlitz polylogarithms, its torsion is expressed in terms of the denominator of Bernoulli-Carlitz numbers and the Fitting ideal of class modules is a special zeta value. We also express the regulator of $\underline{A}(n)$ for positive $n$ in terms of \emph{generalized Carlitz polylogarithms}; after establishing their algebraic relations using difference Galois theory together with the Anderson-Brownawell-Papanikolas criterion, we prove that the regulator is an isomorphism if, and only if, $n$ is prime to the characteristic.
\end{abstract}


{\footnotesize
\tableofcontents
}

\section{Introduction}
\subsection*{Motivations}
Let $\bF$ be a finite field with $q$ elements. The counterpart of the values of the Riemann zeta function at integral arguments in function fields arithmetic, known as \emph{Carlitz's zeta values}, were introduced by Carlitz in \cite{carlitz}. To define them, consider the formal power series in $x$:
\begin{equation}\label{eq:intro-zeta-theta}
\quad Z(x,n):=\sum_{d=0}^{\infty}{S_d(-n) x^d}, \quad \text{where} \quad  S_d(-n):=\sum_{\substack{a~\text{monic} \\ \deg(a)=d}}{\frac{1}{a(\theta)^n}}\in \bF(\theta).
\end{equation}
For positive integers $n$, we define $\zeta_C(n)$ as the converging series $Z(1,n)$ in the field $\KI=\bF(\!(1/\theta)\!)$. For non positive $n$, the sum $S_d(-n)$ vanishes for large enough $d$ (depending on~$n$, e.g. \cite{goss-zeta}), which allows to define $\zeta_C(n)$ by the same formula, now as an element of $\bF[\theta]$. \\

Special values of $\zeta_C$ are arguably equally interesting as those of $\zeta$. For instance, we have the \emph{Carlitz-Euler relation}, proved by Carlitz himself in \cite{carlitz}:
\[
\text{For~all~} n>0,~ q-1|n: \quad \zeta_C(n)=\frac{\operatorname{BC}_n}{\Pi(n)}\tilde{\pi}^n
\]
where $\operatorname{BC}_n\in \bF(\theta)$ is the $n$th \emph{Bernoulli-Carlitz rational fraction}, $\Pi(n)\in \bF[\theta]$ is \emph{Carlitz's factorial of $n$}, and $\tilde{\pi}$ -- which belongs to a separable closure $\KI^s$ of $\KI$ -- is \emph{Carlitz's period}. Such a denomination is justified by \emph{Carlitz's exponential function}: 
\[
\exp_C(x)=x+\frac{x^q}{\theta^q-\theta}+\frac{x^{q^2}}{(\theta^{q^2}-\theta^q)(\theta^{q^2}-\theta)}+\ldots=:\sum_{i=0}^{\infty}{\frac{x^{q^i}}{D_i}}
\]
which converges for any $x$ in $\KI^s$ and whose kernel is $\tilde{\pi}\cdot \bF[\theta]$. Carlitz's factorial $\Pi(n)$ is then defined by $\prod_{i}D_i^{c_i}$ where $n=c_0+c_1q+...$  is the base $q$ expansion of $n$. The Bernoulli-Carlitz rational fractions are defined analogously to classical arithmetic with $\operatorname{BC}_n/\Pi(n)$ standing as the $n$th coefficient in the series expansion of $x/\exp_C(x)$ (see \cite[\S 3 \& \S 9]{goss}). Due to the non-archimedean nature of $\KI$, the value $\zeta_C(1)$ is convergent and we have
\begin{equation}\label{eq:class-formula}
\zeta_C(1)=\log_C(1)
\end{equation}
where $\log_C$ is the \emph{Carlitz logarithm}:
\[
\log_C(x)=x+\frac{x^q}{\theta-\theta^q}+\frac{x^{q^2}}{(\theta-\theta^q)(\theta-\theta^{q^2})}+\ldots=x+\sum_{i=1}^{\infty}{\frac{x^{q^i}}{(\theta-\theta^q)(\theta-\theta^{q^2})\cdots (\theta-\theta^{q^i})}}
\]
defined as local inverse of $\exp_C(x)$. Formula \eqref{eq:class-formula} was again proved by Carlitz in \cite{carlitz}, and generalized by Taelman to his celebrated \emph{class formula} in \cite{taelman-L}. For general positive argument $n>0$, it was proven by Anderson-Thakur \cite{anderson-thakur} that  $\zeta_C(n)$ is expressible as an $\bF[\theta]$-linear combination of values of the \emph{Carlitz $n$th polylogarithm}:
\begin{equation}\label{eq:CarlitzPolylog}
\Li_n(x)=x+\sum_{i=1}^{\infty}{\frac{x^{q^i}}{(\theta-\theta^q)^n(\theta-\theta^{q^2})^n\cdots (\theta-\theta^{q^i})^n}}.
\end{equation}
at some elements $x$ of $\bF[\theta]$ of degree $<\frac{nq}{q-1}$ to ensure convergence. \\

In classical arithmetic, there is a bridge among zeta values and polylogarithms provided by \emph{regulators}. If $F$ is a number field with ring of integers $\cO_F$, Borel introduced an $\bR$-linear map, known as \emph{Borel's regulator}:
\begin{equation}\label{eq:borel-reg}
\operatorname{Reg}_{\operatorname{Borel}}:K_{2n-1}(\cO_F)\longrightarrow \bR^{d_n}
\end{equation}
where $d_n$ is an explicit integer which depends on $F$. On one hand, after careful analysis of functional equations satisfied by polylogarithms, Zagier \cite{zagier} conjectured an explicit presentation for the $K$-group $K_{2n-1}(\cO_F)$ and interpreted \eqref{eq:borel-reg} using a generalization of the Bloch-Wigner dilogarithm. On the other, Borel's regulator \eqref{eq:borel-reg} is known to coincide -- up to a factor $2$ -- with \emph{Beilinson's regulator} associated to the motive $M=\operatorname{Res}_{\cO_F/\bZ}\bQ(n)$ obtained as the restriction from $\cO_F$ to $\bZ$ of $n$th Tate twist over $\cO_F$:
\begin{equation}\label{eq:Beilinson-Reg}
\operatorname{Reg}_{\operatorname{Beilinson}}:\Ext^1_{\bZ}(\mathbbm{1},M)\longrightarrow \Ext^1_{\operatorname{Hdg}^+}(\mathbbm{1},M).
\end{equation}
In general, given a pure motive $M$ over $\bQ$, \emph{the Beilinson regulator of $M$} is the $\bR$-linear map \eqref{eq:Beilinson-Reg}, from the $\bQ$-vector space of extensions of mixed motives over $\bZ$ to $\bR$-vector space of extensions of mixed Hodge structures equipped with an action of the complex conjugation, induced by the exactness of the Hodge realization functor $\cH$. The far reaching \emph{Beilinson's conjecture} predicts the following:
\begin{enumerate}[label=$(\operatorname{BC})_{\arabic*}$]
    \item\label{item:BC1} The $\bQ$-vector space $\Ext^1_{\bZ}(\mathbbm{1},M)$ is finite dimensional,
    \item\label{item:BC2} If the weight of $M$ is $<-2$, the image of $\operatorname{Reg}_{\operatorname{Beilinson}}$ forms a $\bQ$-lattice of full rank,
    \item\label{item:BC3} The covolume of the latter lattice is the \emph{special $L$-value} of $M$.
\end{enumerate}
We refer to the survey \cite{nekovar} for details.  In the situation of $M=\operatorname{Res}_{\cO_F/\bZ}\bQ(n)$ and $n>1$, the latter recovers the Dedekind zeta value at $n$. Although the category of mixed motives with all the required properties as yet to be constructed, this picture makes sense for mixed Tate motives due to the work of Levine \cite{levine}. All together, we recover the link between polylogarithms and zeta values through $K$-theory and motivic cohomology of Tate twists. \\

The function field counterpart motivic cohomology and of Beilinson's regulator \eqref{eq:Beilinson-Reg} was recently introduced by the first author in \cite{gazda2}. The construction is motivated by the strong analogies between the classical category of mixed motives over $\bZ$ and the category of \emph{Anderson $t$-motives over $\bF[\theta]$}. The following definition first appeared in \cite{anderson}: 
\begin{Definition}[Anderson 1986]\label{def:tMot}
A \emph{$t$-motive over $A=\bF[\theta]$} is a pair $\underline{M}=(M,\tau_M)$ where $M$ is a finite free\footnote{In the usual definition of $t$-motives over an arbitrary base $R$, the underlying module $M$ is required to be locally-free of finite constant rank over $R[t]$ (e.g. \cite{hartl-isogeny}). In the situation $R=\bF[\theta]$ (or, latter $R=\bF(\theta)$), this is equivalent to be finite free.} $\bF[\theta,t]$-module and where $\tau_M$ is an isomorphism of $\bF[\theta,t]$-modules:
\[
\tau_M:(\tau^*M)\left[\frac{1}{t-\theta}\right]\stackrel{\sim}{\longrightarrow} M\left[\frac{1}{t-\theta}\right].
\]
Here $\tau:\bF[\theta,t]\to \bF[\theta,t]$ denotes the ring endomorphism mapping $\theta\mapsto \theta^q$ and leaving $\bF[t]$ invariant. It is said that $\underline{M}$ is \emph{effective} if $\tau_M(\tau^*M)\subset M$.
\end{Definition}
We organize $t$-motives over $A$ into a category $\tMot_A$ with obvious morphisms: it is an $\bF[t]$-linear additive category and can be equipped with a closed monoidal structure with a neutral object $\mathbbm{1}=(\bF[\theta,t],\mathbf{1})$ (see Section \ref{sec:tools}). This category is however not abelian as the "free" assumption prevents kernel and cokernel to exist in general. We can nevertheless declare a short sequence $0\to \underline{M}\to \underline{E}\to \underline{N}\to 0$ in $\tMot_A$ to be \emph{exact} whenever the underlying sequence of $\bF[\theta,t]$-modules is. This endows $\tMot_A$ with the structure of an exact category \cite[Prop. 2.15]{gazda1}, and we denote by $\Ext^1_A(\underline{N},\underline{M})$ the corresponding extension module of $\underline{N}$ by $\underline{M}$. \\

The next proposition is classical and is proven in Prop. 2.16 in \emph{loc.\,cit.}.
\begin{Proposition}\label{prop:extensions}
The map $\displaystyle\iota:\Hom_{\bF[\theta,t]}(\tau^*N,M)\left[\frac{1}{t-\theta}\right]\to \Ext^1_{A}(\underline{N},\underline{M})$ given by
\[
h\longmapsto \left[0\to \underline{M}\to \left(M\oplus N,\left(\begin{smallmatrix}\tau_M & h \\ 0 & \tau_N \end{smallmatrix}\right)\right)\to \underline{N}\to 0\right]
\]
is an $\bF[t]$-linear surjection whose kernel is the submodule generated by the differences $f\circ \tau_N-\tau_M\circ \tau^*f$, for $f\in \Hom_{\bF[\theta,t]}(N,M)$.
\end{Proposition}

As is apparent from the proposition, the module $\Ext^1_A(\underline{N},\underline{M})$ is by no means finitely generated over $\bF[t]$ whenever $\underline{M}$ or $\underline{N}$ are both non zero. This is partly due to the fact that, given $h\in \Hom(\tau^*N,M)$ non zero, the sequence of elements $(h\cdot (t-\theta)^{-i})_{i\geq 0}$ produces through $\iota$ extensions that will eventually be linearly independent to all the previous ones. \\
To bound the apparent "pole at $t=\theta$", it has been suggested to restrict attention to \emph{effective extensions}; this leads however to \emph{erratic} results\footnote{Given a symmetric monoidal exact category $\eC$ and two objects $X$ and $Y$, $Y$~being dualizable with dual $Y^{\vee}$, there is an isomorphism functorial in both $X$ and $Y$:
\begin{equation}\label{eq:dualizibility-in-C}
\Ext^1_{\eC}(Y,X)\stackrel{\sim}{\longrightarrow} \Ext^1_{\eC}(\mathbbm{1},X\otimes Y^{\vee})
\end{equation}
where $\mathbbm{1}$ is a neutral object for the monoidal structure. Although the category of effective $t$-motives does not admit duals, the isomorphism \eqref{eq:dualizibility-in-C} may fail even if both sides make sense. For instance, it was realized by Mornev and Taelman that the module of effective extensions of $\underline{A}(n)$ by $\underline{A}(m)$ does not only depend on the difference $m-n$, thus preventing an isomorphism like \eqref{eq:dualizibility-in-C} to hold.}. We rather follow Pink's insight in \cite[\S 8]{pink}, and consider \emph{regulated extensions} as is done in \cite{gazda1} and \cite{gazda2}.\\

Before stating the definition of \emph{regulation}, we recall that of \emph{Hodge polygons}. Consider~$M_K=M\otimes_{\bF[\theta]}\bF(\theta)$. The relative elementary divisors of $\tau_M(\tau^*M_K)$ and $M_K$ locally at the ideal $(t-\theta)$ of $\bF(\theta)[t]$ define integers $w_1\leq w_2\leq \cdots \leq w_r$ where $r$ is the rank of $M$. The \emph{Hodge polygon of $\underline{M}$} is the unique piece-wise linear convex function $[0,r]\to \bR$ starting at $(0,0)$ and having a given slope $s\in \bZ$ on a subinterval of length $\#\{0\leq i\leq r|w_i=s\}$.

\begin{Definition}\label{def:regulated}
An exact sequence (or extension) $0\to \underline{M}\to \underline{E}\to \underline{N}\to 0$ in $\tMot_A$ is \emph{regulated} if the Hodge Polygons of $\underline{M}\oplus \underline{N}$ and $\underline{E}$ coincide. 
\end{Definition}

Observe that the analogous condition for number fields is always satisfied as the Hodge realization functor is exact. In the function field situation, one can show that this notion already solves the "boundedness" issue (\emph{cf}. \cite[Prop. 5.3]{gazda1}):
\begin{Proposition}\label{prop:regulation}
Let $h\in \Hom(\tau^*N,M)[(t-\theta)^{-1}]$. The extension $\iota(h)$ is regulated if and only if $h=g\circ \tau_N-\tau_M\circ f$ for some linear maps $g:N\to M$ and $f:\tau^*N\to \tau^*M$.
\end{Proposition}

The submodule of \emph{regulated extensions} $\Ext^{1,\operatorname{reg}}_A(\underline{N},\underline{M})$ is however still not finitely generated in general. To understand why, we now assume that both $\underline{N}$ and $\underline{M}$ are objects of $\tMotrat_A$, the full subcategory of $\tMot_A$ consisting of \emph{rigid analytically trivial} objects (Definition \ref{def:rat}). To extract a natural finite submodule out of $\Ext^{1,\operatorname{reg}}_A(\underline{N},\underline{M})$, one introduces the $A$-motivic \emph{Betti realization} as suggested\footnote{A similar strategy was employed by Taelman \cite{taelman} in the case of Drinfeld modules.} in \cite{gazda2}: similar to the number field case, there is a functor $\underline{M}\mapsto \underline{M}_B$ from $\tMotrat_A$ to the category of $\bF[t]$\nobreakdash-linear continuous representations of the profinite group $G_{\infty}=\Gal(\KI^s|\KI)$, playing the role of the classical singular realization. Its exactness induces a morphism of $\bF[t]$-modules:
\begin{equation}\label{eq:rBetti}
r_B(\underline{N},\underline{M}):\Ext^{1,\operatorname{reg}}_A(\underline{N},\underline{M})\longrightarrow \Ext^1_{A[G_\infty]}(\underline{N}_B,\underline{M}_B).
\end{equation}
By contrast with the classical situation where coefficients are in $\bQ$ and where $\Gal(\bC|\bR)$ is the counterpart of $G_{\infty}$, the right-hand side of \eqref{eq:rBetti} might be non trivial. It is usual to give individual names to the kernel and cokernel of $r_B$:
\[
\Ext^{1,\operatorname{reg}}_{A,\infty}(\underline{N},\underline{M}):=\ker r_B(\underline{N},\underline{M}), \quad \operatorname{Cl}(\underline{M}):=\coker r_B(\mathbbm{1},\underline{M}).
\]
In most of what follows, $\underline{N}=\mathbbm{1}$, in which case we call them respectively \emph{the integral $t$-motivic cohomology of $\underline{M}$} and \emph{the class module of $\underline{M}$}. The next theorem was the gist of \cite{gazda2} (see Thm. 4.1 there):
\begin{Theorem}
The modules $\Ext^{1,\operatorname{reg}}_{A,\infty}(\underline{N},\underline{M})$ and $\operatorname{Cl}(\underline{M})$ are finitely generated over $\bF[t]$.
\end{Theorem}
\begin{Remark}\label{rmk:dual}
The theorem in \emph{loc.\,cit.} was in fact stated for $\underline{N}=\mathbbm{1}$. The above version is equivalent to it, as the canonical morphism $\Ext^1_{A}(\underline{N},\underline{M})\to \Ext^1_{A}(\mathbbm{1},\underline{M}\otimes \underline{N}^{\vee})$ maps an extension $\underline{E}$ of $\underline{N}$ by $\underline{M}$ to a regulated extension of $\mathbbm{1}$ by $\underline{N}^{\vee}\otimes \underline{M}$ (resp. having analytic reduction at $\infty$) if and only if $\underline{E}$ is itself regulated (resp. has analytic reduction at $\infty$) (see \cite[Cor. 5.4]{gazda1}). In particular, it induces an isomorphism
\[
\Ext^{1,\operatorname{reg}}_{A,\infty}(\underline{N},\underline{M}) \stackrel{\sim}{\longrightarrow}\Ext^{1,\operatorname{reg}}_{A,\infty}(\mathbbm{1},\underline{M}\otimes \underline{N}^{\vee}).
\]
\end{Remark}

In \cite[Def. 4.3]{gazda2}, a regulator attached to a $t$-motive is introduced, aiming to play the part of Beilinson's regulator. The counter-part of Hodge structures in the context of Anderson $t$-motives were discovered by Pink in \cite{pink} and are called \emph{Hodge-Pink structures}. Similarly, there is a \emph{Hodge-Pink realization functor} $\cH$ having the category $\tMotrat$ for source, and the regulator is defined as the induced regulator morphism on extensions:
\begin{equation}\label{eq:my-reg}
\operatorname{Reg}(\underline{M}):\Ext^{1,\operatorname{reg}}_{A,\infty}(\bbone,\underline{M}) \longrightarrow \Ext^{1,\operatorname{ha}}_{\mathbf{HP}^+,\infty}(\bbone,\cH(\underline{M})),
\end{equation}
where the target is a certain $\KI$-vector space of extensions of Hodge-Pink structures. We refer the reader to \cite{gazda2} for notations and constructions. The following was the second main result of \emph{loc.\,cit.} (see Thm. 4.4 there):
\begin{Theorem}
Assume that the weights of $\underline{M}$ are negative. Then the rank of source of \eqref{eq:my-reg} as an $A$-module equals the dimension of its target as a $\KI$-vector space.  
\end{Theorem}

However, it is not known whether the image of $\operatorname{Reg}(\underline{M})$ has full rank in general, namely, if the counterpart of Beilinson's conjecture \ref{item:BC2} holds for $\underline{M}$. Below, we shall show that this is not always the case, even for $t$-motives \emph{as simple as} Carlitz twists.

\subsection*{Results}
This paper is devoted to the determination of the integral $t$-motivic cohomology and of the class module in the case where $\underline{M}$ is Carlitz's $n$th twist $\underline{A}(n)$ over $\bF[\theta]$, the counterpart of the classical Tate twist $\bQ(n)$ over $\bZ$:
\begin{Definition}\label{def:ct}
The \emph{$n$th Carlitz twist} $\underline{A}(n)$ is the $t$-motive over $\bF[\theta]$ whose underlying module is $\bF[\theta,t]$ and whose morphism is $(t-\theta)^{-n}\mathbf{1}$.
\end{Definition}

In virtue of Proposition \ref{prop:extensions}, the extension module $\Ext^1_{A}(\mathbbm{1},\underline{A}(n))$ is given as the image of $\bF[\theta,t][(t-\theta)^{-1}]$ through $\iota$. We define 
\begin{Definition}
For $\alpha\in \bF[\theta]$ of degree $<\frac{nq}{q-1}$, we call \emph{the polylog extension class of $\alpha$} and denote by $[\underline{L}(\alpha)]$ the extension obtained as the image of $\frac{\alpha}{(t-\theta)^n}$ by $\iota$ in $\Ext^1_{A}(\mathbbm{1},\underline{A}(n))$.
\end{Definition}

The naming \emph{polylog class} is justified by the fact that Carlitz polylogarithms appear as function fields periods of the underlying $t$-motive (e.g. \cite[\S 4.3]{chang} in the setting of \emph{dual $t$-motives}). Our first result describes the structure of the integral $t$-motivic cohomology of the Carlitz's $n$th twist $\underline{A}(n)$.

\begin{MainTheorem}\label{thm:polylogclass}
The polylog classes of the elements in $\{\alpha\in \bF[\theta]|\deg_\theta \alpha<n\}$ generate the $\bF[t]$-module $\Ext^{1,\operatorname{reg}}_{A,\infty}(\mathbbm{1},\underline{A}(n))$. Furthermore, there is a relation of the form
\[
a_1(t)\cdot [\underline{L}(\alpha_1)]+\ldots+a_s(t)\cdot [\underline{L}(\alpha_s)]=0, \quad a_i(t)\in \bF[t],
\]
if, and only if, $a_1(\theta)\Li_n(\alpha_1)+\ldots+a_s(\theta)\Li_n(\alpha_s)$ is an $\bF[\theta]$-multiple of $\tilde{\pi}^n$. 
\end{MainTheorem}

\begin{Remark}
Very recently, Chen-Harada \cite{chen-harada} obtained linear relations among Carlitz polylogarithms using \emph{multiplicative dependence} in the context of the Carlitz module and its tensor powers. It would be interesting to understand the relation among our approaches. 
\end{Remark}

The way the theorem is stated is somehow reminiscent of Zagier's conjecture for number fields; the latter predicts that linear relations among classical $n$th polylogarithms are dictated by relations of classes in the group $K_{2n-1}(F)$ (for $F$ number fields), itself related to the motivic extensions of $\mathbbm{1}$ by $\bQ(n)$ (see \cite{beilinson-deligne} for details). 

We also derive the following explicit presentation of the integral $t$-motivic cohomology (see Theorems \ref{thm:twist-zero}, \ref{thm:revelo} and \ref{thm:positive-twists}):

\begin{MainTheorem}\label{thm:mot-coh-car}
For all integers $n$, we have isomorphisms of $\bF[t]$-modules:
\[
\Ext^{1,\operatorname{reg}}_{A,\infty}(\mathbbm{1},\underline{A}(n)) \cong \begin{cases} (0) & \text{if~}n\leq 0, \\ \bF[t]^{n} & \text{if~}n>0 \text{~ and~ }q-1\nmid n, \\ \bF[t]^{n-1}\oplus \bF[t]/(\varepsilon_n(t)) & \text{if~}n>0 \text{~ and~ }q-1|n,  \end{cases}
\]
where, for $n>0$ and divisible by $q-1$, the polynomial $\varepsilon_n$ is determined as follows. Let $m=\frac{n}{q-1}$, and write $m=p^c m_0$ for $c\geq 0$ and $m_0>0$ prime to $p$. Let also $\ell$ be the least common multiple of the integers $r$ such that $q^r-1|n$. Then, we have
\[
\varepsilon_n(t)=\left(t^{q^\ell}-t \right)^{p^c}.
\]
\end{MainTheorem}

\begin{Example*}
If $n=q^k(q-1)$ for some integer $k\geq 0$, then $\varepsilon_n(t)=t^{q^{k+1}}-t^{q^k}$. This comes from the identity
\[
\Li_n(1)=\frac{\tilde{\pi}^n}{\theta^{q^{k+1}}-\theta^{q^k}}.
\]
However, if $n$ is not of this form, relations among $\tilde{\pi}^n$ and Carlitz polylogarithms may be subtler. Pellarin recently found new identities of this type \cite{pellarin}; it would be interesting to interpret them from the point of view of integral $t$-motivic cohomology.   
\end{Example*}

Let $\{\alpha_1,\ldots,\alpha_n\}$ be a basis of $\{\alpha\in\bF[\theta]|\deg_\theta \alpha<n\}$. The combination of Theorems \ref{thm:polylogclass} and \ref{thm:mot-coh-car} implies in particular that the family $(\Li_n(\alpha_1),\ldots,\Li_n(\alpha_n))$ is linearly independent over $\bF[\theta]$ when $q-1\nmid n$, and linearly dependent to $\tilde{\pi}^n$ if $q-1|n$. Stronger results hold (see Theorem \ref{thm:algebraic-independence-of-higher-Carlitz-polylog} below).

\begin{Remark}
Contrary to $K_{2n-1}(\bZ)$, the rank of $\Ext^{1,\operatorname{reg}}_{A,\infty}(\mathbbm{1},\underline{A}(n))$ \emph{grows as $n$ grows}. This fact is not erratic, but rather a feature of function fields arithmetic. The reason is \emph{morally} given as follows: for $n\geq 1$, the disk of convergence of the Carlitz $n$th polylogarithm \eqref{eq:CarlitzPolylog} is $\{x\in \KI|\deg(x)<nq/(q-1)\}$. In particular, its radius is not $1$ -- as for the number field polylogarithm -- but slightly bigger and grows strictly as $n$ grows. For each element $\alpha\in \bF[\theta]$ in the radius of convergence of $\Li_n(x)$, there is an associated polylog class $[L_n(\alpha)]$ in the integral $t$-motivic cohomology of $\underline{A}(n)$ by Theorem \ref{thm:polylogclass}. 
\end{Remark}

D. Thakur, in \cite{thakur2004}, suggested that the function field analogue of $\operatorname{B}_n/n$ -- where $\operatorname{B}_n$ is the $n$th Bernoulli number (e.g. $\operatorname{B}_1=-\frac{1}{2}$, $\operatorname{B}_2=\frac{1}{6}$, $\operatorname{B}_3=0$) -- is rather the rational fraction $\operatorname{BC}_n \cdot \Pi(n-1)/\Pi(n)$. The following amusing corollary of Theorem \ref{thm:mot-coh-car}, obtained from explicit computations due to Alejandro-Rodríguez \cite{alejandro-rodriguez}, relates the torsion of the integral $t$-motivic cohomology of $\underline{A}(n)$ to the denominator of those elements:
\begin{Corollary*}
Let $\varepsilon_n$ be as in Theorem \ref{thm:mot-coh-car}. Then $\varepsilon_n(\theta)\cdot \operatorname{BC}_n \cdot \Pi(n-1)/\Pi(n)$ is in $\bF[\theta]$.
\end{Corollary*}

That is, the torsion of $t$-motivic cohomology is killed by the denominator of $\operatorname{BC}_n \cdot \Pi(n-1)/\Pi(n)$ written in lowest term. This is somehow similar to the classical situation where the the torsion in $K_{2n-1}(\bZ)$ is killed by the denominator of $\operatorname{B}_n/n$.\\

On the other-hand, the class module $\operatorname{Cl}(\underline{A}(n))$ seems to play a role \emph{dual} to that of $\Ext^{1,\operatorname{reg}}_{A,\infty}(\mathbbm{1},\underline{A}(n))$, as we next show (see Theorems \ref{thm:twist-zero}, \ref{thm:class-module-n-neg} and also Proposition \ref{prop:class-module-n-pos}):
\begin{MainTheorem}\label{thm:rk-Cl}
For all integers $n$, we have isomorphisms of $\bF[t]$-modules:
\[
\operatorname{Cl}(\underline{A}(n)) \cong \begin{cases} (0) & \text{if~}n\geq 0, \\ \operatorname{Cl}(\underline{A}(n))_{\operatorname{tors}} & \text{if~}n<0 \text{~ and~ }q-1\nmid n, \\ \bF[t]\oplus \operatorname{Cl}(\underline{A}(n))_{\operatorname{tors}} & \text{if~}n>0 \text{~ and~ }q-1|n,  \end{cases}
\]
where $\operatorname{Cl}(\underline{A}(n))_{\operatorname{tors}}$ is the torsion submodule of $\operatorname{Cl}(\underline{A}(n))$.
\end{MainTheorem}

Determining the torsion module $\operatorname{Cl}(\underline{A}(n))_{\operatorname{tors}}$ seems to be a challenging task. However, we are able to describe its Fitting ideal. Reconsider the series $Z(x,n)$ introduced earlier in \eqref{eq:intro-zeta-theta}, although this time in the variable $t$ instead of $\theta$, as we now interpret zeta values in the coefficient ring $\bF[t]$. Let us write $Z(x,n)$ uniquely as 
\[
Z(x,n)=\zeta_C^*(n)(x-1)^{h_n}+O((x-1)^{h_n+1}) \in \bF[t][\![x-1]\!],
\]
where $h_n\geq 0$ and $\zeta_C^*(n)\in \bF[t]$. Using the trace formula in the form described by Taelman in \cite{taelman-woodshole}, we obtain (see Theorem \ref{thm:zeta}):  

\begin{MainTheorem}\label{thm:zeta?}
For all integers $n$, the rank of $\operatorname{Cl}(\underline{A}(n))$ is equal to $h_n$. For non positive~$n$, $\zeta_C^*(n)$ generates the Fitting ideal of $\operatorname{Cl}(\underline{A}(n))$. 
\end{MainTheorem}

We believe that Theorem \ref{thm:zeta?} could be explained by elementary means. In Remark \ref{rmj:elementary-strategy} below, we give an elementary strategy to tackle this result, and challenge the interested reader to solve it this way. \\

As an application of Theorems \ref{thm:mot-coh-car}, \ref{thm:zeta?}, and together with classical results on Carlitz zeta values due to Carlitz, Lee and Thakur, we obtained the following result which describes situations when the module of regulated extensions is computed solely in terms of Galois cohomology (see Corollary \ref{cor:thakur}):
\begin{Corollary*}
Assume that $q=p$ is prime and let $n\geq 0$. Then, the map 
\[
r_B:\Ext^{1,\operatorname{reg}}_A(\mathbbm{1},\underline{A}(-n))\longrightarrow \operatorname{H}^1(G_\infty,\underline{A}(-n)_B),
\]
deduced from \eqref{eq:rBetti} with $\underline{N}=\mathbbm{1}$ and $\underline{M}=\underline{A}(-n)$, is an isomorphism if, and only if, the sum of the digits of $n$ written in base $p$ is $<p-1$.
\end{Corollary*}

Our last task is devoted to the study of the regulator of $\underline{A}(n)$ and the counterpart of \ref{item:BC2}. To express the regulator, we introduce variants of the Carlitz polylogarithm, defined as follows. For $x\in \KI^s$ with $\deg x<\frac{nq}{q-1}$, the series 
\begin{equation}\nonumber
\xi_{x,n}(t)=\sum_{i=0}^{\infty}{\frac{x^{q^i}}{(t-\theta)^n(t-\theta^q)^n\cdots (t-\theta^{q^i})^n}}
\end{equation}
converges to a meromorphic function in the variable $t\in \KI^s$ having a pole of order $n$ at $t=\theta$. This series, which has already been introduced by Chang-Yu \cite[(3.3)]{chang-yu}, can be suitably interpreted as a \emph{deformation} of $\Li_n(x)$. This allows to introduce \emph{generalized Carlitz polylogarithms} $\Li_n^{[j]}(x)\in \KI^s$ by mean of the expansion: 
\begin{equation}\label{eq:higherpolylog}
\xi_{x,n}(t)=\frac{\Li_n^{[0]}(x)}{(t-\theta)^n}+\frac{\Li_n^{[1]}(x)}{(t-\theta)^{n-1}}+\ldots+ \frac{\Li_n^{[n-1]}(x)}{(t-\theta)}+O(\KI^s[\![t-\theta]\!]).
\end{equation}
In particular, $\Li_n^{[0]}(x)$ is the Carlitz $n$th polylogarithm defined in \eqref{eq:CarlitzPolylog}. Using \emph{prolongation $t$-motives} as defined by the second author in \cite{maurischat}, difference Galois theory and the Anderson-Brownawell-Papanikolas criterion, we are able to show:
\begin{MainTheorem}\label{thm:algebraic-independence-of-higher-Carlitz-polylog}
Assume that $n$ is prime to the characteristic $p$ of $\bF$. Let $(\alpha_1,\ldots,\alpha_s)$ be elements of $\bF[\theta]$ of degree $<n$ such that the family of classes $\left([\underline{L}(\alpha_1)],\ldots,[\underline{L}(\alpha_s)]\right)$ are linearly independent in 
$\Ext^{1,\operatorname{reg}}_{\infty}(\bbone,\underline{A}(n))$. Then, the family of elements in $\KI^s$,
\[
(\tilde{\pi},\Li_n^{[j]}(\alpha_i),~i\in\{1,\ldots,s\},~j\geq 0)
\]
is algebraically independent over $\bF(\theta)$.
\end{MainTheorem}

Our main interest in Theorem \ref{thm:algebraic-independence-of-higher-Carlitz-polylog} stands in the next description of the regulator map. We show in Proposition \ref{prop:equivalence-beilinson-for-carlitz} below:
\begin{Proposition*}
Let $(\alpha_1,\ldots,\alpha_s)$ be elements of $\bF[\theta]$ of degree $<n$ such that the family of classes $\left([\underline{L}(\alpha_1)],\ldots,[\underline{L}(\alpha_s)]\right)$ is a basis of 
$\Ext^{1,\operatorname{reg}}_{\infty}(\bbone,\underline{A}(n))$. Then, $\mathrm{Reg}(\underline{A}(n))$ is represented by the matrix 
\[
\begin{pmatrix}
\Li_n^{[n-s]}(\alpha_1) & \Li_n^{[n-s]}(\alpha_2) & \cdots & \Li_n^{[n-s]}(\alpha_s) \\ 
\Li_n^{[n-s+1]}(\alpha_1) & \Li_n^{[n-s+1]}(\alpha_2) & \cdots & \Li_n^{[n-s+1]}(\alpha_s) \\ 
\vdots & \vdots & \ddots & \vdots \\
\Li_n^{[n-1]}(\alpha_1) & \Li_n^{[n-1]}(\alpha_2) & \cdots & \Li_n^{[n-1]}(\alpha_s) 
\end{pmatrix} \in \Mat_s(\KI).
\]
\end{Proposition*}

In particular, the combination of Theorem \ref{thm:algebraic-independence-of-higher-Carlitz-polylog} and the above proposition gives the following counterpart of Beilinson's conjecture:
\begin{Corollary*}
If $p\nmid n$, the regulator of $\underline{A}(n)$ is an isomorphism. 
\end{Corollary*}

In retrospect, this is surely the most sophisticated way we ever encountered to prove the non-vanishing of a determinant! \\

However, the above proposition also implies that the conclusion of this corollary does not hold whenever $p|n$. From the relation $\xi_{\alpha,n}(t)^p=\xi_{\alpha^p,np}(t)$, where $p$ is the characteristic of $\bF$, we have
\begin{equation}\label{eq:p-polylog}
\text{For~}k\geq 0:\quad \Li_{np}^{[k]}(\alpha^p)=\begin{cases}0 & \text{if~}p\nmid k, \\ \Li_n^{[k/p]}(\alpha)^p & \text{if~}p\mid k \end{cases}
\end{equation}
which implies the vanishing of some lines in the above matrix. Concretely, one derives that the rank of $\mathrm{Reg}(\underline{A}(n))$ is $\frac{n}{p^{v_p(n)}}-\mathbf{1}_{q-1|n}$ (see Corollary \ref{cor:exact-rank}). This is quite a surprise and contrasts with the number field situation. For general $t$-motives $\underline{M}$, it is not clear -- at least to the authors -- what to expect regarding the rank of $\mathrm{Reg}(\underline{M})$.

\paragraph{Acknowledgments:} We are much obliged to Steven Charlton for his computational devotion to a late conjecture which evolved into Theorem \ref{thm:zeta?}, and to Dinesh Thakur for his multiple questions and suggestions regarding this research. We are also very much indebt to Javier Fres\'an for both his lights on the number fields situation, as well as his precious proofreading of the present document. The first author is grateful to Max Planck Institute for Mathematics in Bonn for its hospitality and financial support.

\section{Tools and notations from the general theory}\label{sec:tools}
\subsection{Notation}
Let $\bF$ be a finite field with $q$ elements and of characteristic $p$. Let $A=\bF[\theta]$, and $K=\bF(\theta)$ its fraction field. There are in fact two isomorphic versions of $K$ (or $A$) playing a role in this paper: one for the \emph{coefficient field} and one for the \emph{base field}. As it is often done in the literature, we will distinguish two variables $t$ and $\theta$ for \emph{coefficients} and \emph{base} respectively.

As an instance of this, we denote by $K[t]$ the polynomial ring in the variable $t$ with coefficients in $K$, and by $\tau:K[t]\to K[t]$ the ring endomorphism raising the coefficients to the $q$th power. Still in accordance with the literature, we also denote by $p^{(k)}$ the element $\tau^k(p)$ and call it \emph{the $k$th twist of $p$} (for $p\in K[t]$ and $k\geq 1$).

\medskip

On $K$, we consider the norm $\abs{\cdot}=q^{\deg(\cdot)}$, and we denote by $\KI$ the completed field, identified with $\bF(\!(\theta^{-1})\!)$. The norm $\abs{\cdot}$ extends to a separable closure $\KI^s$ of $\KI$. The completion of $\KI^s$, denoted by $\CI$, is well-known to be algebraically closed and complete. 

\medskip

For $L\subset \CI$ a complete subfield, we denote by $L\langle t \rangle$ the \emph{Tate algebra} over $L$, that is, the $L$-algebra of power series in the variable $t$ whose coefficients tend to $0$:
\[
L\langle t \rangle =\left\{\sum_{i=0}^\infty a_i t^i~\bigg |~\text{for~all~} i:~a_i\in L,~\lim_{i\to \infty}\abs{a_i}=0 \right\}.
\]
For $f=\sum_i{a_i t^i}\in L\langle t \rangle$, we still denote by $\tau(f)=f^{(1)}$ its \emph{twist} $\sum_i{a_i^q t^i}\in L\langle t \rangle$. We also equip $L\langle t \rangle$ with the \emph{Gauss norm} $\norm{\cdot}$ defined by means of the formula
\[ \norm{\sum_{i=0}^\infty a_i t^i} := \max\{ \abs{a_i} \mid i\geq 0 \}. \]
The Tate algebra $L\langle t \rangle$ is complete with respect to $\|\cdot \|$, and $L[t]$ is a dense subring. Note that we also have $\norm{f^{(1)}}=\norm{f}^q$. \\

We also let $L\langle\!\langle t \rangle\!\rangle$ be the ring of \emph{entire series} on $L$, that is :
\begin{equation}\label{eq:entire-series}
L\langle\!\langle t \rangle\!\rangle:=\left\{\sum_{i\geq 0}{a_i t^i}\in L[\![t]\!]~|~\forall \rho>0:~\lim_i |a_i|\rho^i=0\right\}.
\end{equation}
It is a subring of $\CI\langle t\rangle$ which is stable under $\tau$ and hyperdifferentiation of any order.

\subsection{Anderson $t$-motives}
Let $R$ be either be $A$ or $K$. We recall the definition of a (non necessarily effective) \emph{Anderson $t$-motive over $R$} : it is a pair $\underline{M}=(M,\tau_M)$ where $M$ is a finite free $R[t]$\nobreakdash-module and where $\tau_M$ is an isomorphism of $R[t]$-modules:
\[
\tau_M:(\tau^*M)\left[\frac{1}{t-\theta}\right]\stackrel{\sim}{\longrightarrow} M\left[\frac{1}{t-\theta}\right].
\]
Recall that $\tau^*M$ is the $R[t]$-module which, as an abelian group, is  $R[t]\otimes_{\tau,R[t]}M$ - where $R[t]$ on the left-hand side of the tensor product is seen as a $R[t]$-module through $\tau$ - and where the $R[t]$-module structure is obtained by multiplying elements on the left-hand side of the tensor product. For $m\in M$, we denote by $\tau^*m\in \tau^*M$ the element $(1\otimes_\tau m)$; in particular, the map
\[
\tau^*:M\longrightarrow \tau^*M, \quad m\longmapsto \tau^* m
\]
is not linear, but only \emph{$\tau$-linear}.

\medskip

A morphism of $t$-motives $f:\underline{M}\to \underline{N}$ over $R$ is an $R[t]$-linear map of the underlying modules such that $f\circ \tau_M=\tau_N\circ \tau^*f$. We gather $t$-motives and their morphisms in an $\bF[t]$-linear category $\tMot_R$. We call a sequence $0\to \underline{M}\to \underline{E} \to \underline{N}\to 0$ in $\tMot_R$ \emph{exact} if its underlying sequence of modules is. It is not hard to check that this equip $\tMot_R$ with the structure of an exact category \cite[Prop. 2.13]{gazda1}.\\

We also endow $\tMot_R$ with a symmetric monoidal structure as it is usually done. Given two $t$-motives $\underline{M}$ and $\underline{N}$, we define their tensor product $\underline{M}\otimes \underline{N}$ as the $t$-motive over $R$ having for module $M\otimes_{R[t]} N$ and for morphism:
\begin{multline*}
\tau^*(M\otimes_{R[t]} N)\left[\frac{1}{t-\theta}\right]\cong (\tau^*M)\left[\frac{1}{t-\theta}\right]\otimes_{R[t]}  (\tau^*N)\left[\frac{1}{t-\theta}\right] \\
\xrightarrow{\tau_M\otimes \tau_N} M\left[\frac{1}{t-\theta}\right]\otimes_{R[t]} N\left[\frac{1}{t-\theta}\right]\cong (M\otimes_{R[t]} N)\left[\frac{1}{t-\theta}\right].
\end{multline*}
One shows that $\tMot_R$ is \emph{rigid} in the sense that it possesses all internal $\Hom$. In particular, all objects are dualizable. \\

Following Anderson \cite{anderson}, we define the \emph{Betti realization of $\underline{M}$} as the $\bF[t]$-module:
\[
\underline{M}_B:=\left\{f\in M\otimes_{R[t]}\CI\langle t \rangle~|~\tau_M(\tau^*f)=f\right\}.
\]
It is well-known to be a finite free $\bF[t]$-module. By continuity, $\underline{M}_B$ carries a compatible action of $G_{\infty}=\Gal(\KI^s|\KI)$ as acting on $\CI\langle t\rangle$ coefficient-wise, and one can show that this action is continuous (see \cite[Prop. 3.11.]{gazda2}). As in \cite{gazda2}, we will make use of the submodule $\underline{M}^+_B$ of elements fixed by $G_\infty$. \\

The following definition goes back to Anderson's original paper:
\begin{Definition}\label{def:rat}
The $t$-motive $\underline{M}$ is said to be \emph{rigid analytically trivial} if the multiplication map:
\[
\underline{M}_B\otimes_{\bF[t]} \CI\langle t \rangle \longrightarrow M\otimes_{R[t]}\CI\langle t \rangle
\]
is bijective.
\end{Definition}
We denote by $\tMotrat_R$ the full subcategory of rigid analytically trivial objects. It is an exact subcategory stable by extensions \cite[\S 3]{gazda2} which, since the functor $\underline{M}\mapsto \underline{M}_B$ also commutes with $\otimes$, is again symmetric monoidal and rigid. \\

Along this paper, we will be interested in \emph{Carlitz twists} $\underline{A}(n)$ as defined in Definition \ref{def:ct}. The notation is here to stress that $\underline{A}(n)$ is the function field counterpart of the Tate twist $\bQ(n)$. In fact, the same additive relations hold:
\[
\underline{A}(0)=\mathbbm{1}, \quad \underline{A}(n)^{\vee}\cong \underline{A}(-n), \quad \underline{A}(n)\otimes \underline{A}(m)\cong \underline{A}(n+m).
\]

Let $\lambda_\theta\in \KI^s$ be a $(q-1)$th root of $-\theta$. The \emph{Anderson-Thakur function $\omega$} is defined as the convergent product in the Tate algebra:
\[
\omega(t)=\lambda_\theta\prod_{i=0}^{\infty}{\left(1-\frac{t}{\theta^{q^i}}\right)^{-1}}\in \KI(\lambda_\theta)\langle t \rangle.
\]
It is well-known that all Carlitz twists are rigid analytically trivial, with Betti realization being of rank $1$ over $\bF[t]$ with generator $\omega^n$, i.e.,
\begin{equation}\label{eq:functional-equation-omega}
\bF[t]\cdot \omega^n =  \left\{ f\in \CI\cs{t} \,\middle|\, (t-\theta)^n f-f^{(1)}=0 \right\} . 
\end{equation}

\subsection{Integral $t$-motivic cohomology and class group}
Let $\underline{M}$ be a rigid analytically trivial $t$-motive over $A=\bF[\theta]$. The exactness of the Betti realization functor \cite[Cor. 3.17]{gazda2} induces an $\bF[t]$-linear morphism
\begin{equation}\label{eq:rB}
\Ext^{1,\operatorname{reg}}_A(\mathbbm{1},\underline{M})\longrightarrow \operatorname{H}^1(G_\infty,\underline{M}_B).
\end{equation}
The following definitions were made in \cite{gazda2}:
\begin{Definition}
\begin{enumerate}[label=$-$]
\item An extension of $\mathbbm{1}$ by $\underline{M}$ has \emph{analytic reduction at $\infty$} if it splits as a continuous representation of $G_\infty$; i.e. belongs to the kernel of \eqref{eq:rB}. We denote by $\Ext^{1}_{A,\infty}(\mathbbm{1},\underline{M})$ this kernel and by $\Ext^{1,\operatorname{reg}}_{A,\infty}(\mathbbm{1},\underline{M})$ the submodule of regulated extensions.
\item The cokernel of \eqref{eq:rB}, denoted $\operatorname{Cl}(\underline{M})$, is called \emph{the class module of $\underline{M}$}.
\end{enumerate}
\end{Definition}

The main tool to access these two modules is a long exact sequence which we recall next. Let $G_{\underline{M}}$ be the following complex of $\bF[t]$-modules sitting in degrees $\{0,1\}$:
\begin{equation}\label{eq:GM}
G_{\underline{M}}=\left[\frac{M\otimes_{A[t]} \KI\langle t \rangle}{M}\xrightarrow{\id-\tau_M} \frac{M\otimes_{A[t]} \KI\langle t \rangle}{M+\tau_M(\tau^*M)} \right]
\end{equation}
and where the boundary map is induced from $m\mapsto m-\tau_M(\tau^*m)$. This complex was introduced in \cite[Def. 4.5]{gazda2} as a main ingredient in the proof of the finite generation of $\Ext^{1,\operatorname{reg}}_{A,\infty}(\mathbbm{1},\underline{M})$ and $\operatorname{Cl}(\underline{M})$. The next proposition was proved in Prop. 4.6, \emph{loc.\,cit.}; we extract the statement of interest from the latter reference and its proof:
\begin{Proposition}\label{prop:description-rB}
There is a long exact sequence of $\bF[t]$-modules:
\begin{equation*}
\underline{M}_B^+\longrightarrow \operatorname{H}^0(G_{\underline{M}}) \stackrel{b}{\longrightarrow} \Ext^{1,\operatorname{reg}}_A(\mathbbm{1},\underline{M}) \stackrel{a}{\longrightarrow} \operatorname{H}^1(G_\infty,\underline{M}_B) \longrightarrow \operatorname{H}^1(G_{\underline{M}})\longrightarrow 0
\end{equation*}
where $a$ denotes the map \eqref{eq:rB}. In particular, $\operatorname{Cl}(\underline{M})$ is isomorphic to $\operatorname{H}^1(G_{\underline{M}})$. In addition, the composition $a\circ b$ maps the class of a representative $\xi\in M\otimes_{A[t]}\KI\langle t \rangle$ to the class of the cocycle $G_\infty\to \underline{M}_B$, $\sigma\mapsto \,^{\sigma}\xi-\xi$. 
\end{Proposition}

We make profit of the latter sequence to derive explicit formulas for the integral $t$-motivic cohomology in the situation of Carlitz twists. Let $n\in \bZ$ and let $\Delta=\Delta_n:\KI\langle t \rangle\to \KI\langle t \rangle$ be the $\bF[t]$-linear $\tau$-difference operator:
\begin{equation}\label{eq:Delta}
\Delta f(t)= f(t)-\frac{f(t)^{(1)}}{(t-\theta)^n}. 
\end{equation}
\begin{Corollary}\label{cor:formula-for-carlitz}
There is a natural isomorphism of $\bF[t]$-modules:
\[
\Ext^{1,\operatorname{reg}}_{A,\infty}(\bbone,\underline{A}(n))\cong \frac{\{\xi\in \KI\langle t \rangle ~|~\Delta \xi \in (t-\theta)^{\min\{0,-n\}}\bF[\theta,t]\}}{\bF[\theta,t]+\mathbf{1}_{q-1|n}\bF[t]\omega(t)^n},
\]
where $\mathbf{1}_{q-1|n}$ equals $1$ if $q-1$ divides $n$, and $0$ otherwise. For $n> 0$, the rank of $\Ext^{1,\operatorname{reg}}_{A,\infty}(\bbone,\underline{A}(n))$ is $n$ if $q-1\nmid n$ and is $n-1$ otherwise.
\end{Corollary}
\begin{proof}
Set $\underline{M}=\underline{A}(n)$. We have $\underline{A}(n)_B=\bF[t]\omega(t)^n$ where $\omega(t)^n\in \KI\langle t \rangle$ if, and only if, $q-1|n$. That is, $\underline{A}(n)_B^+=\mathbf{1}_{q-1|n}\bF[t]\omega(t)^n$. On the other-hand, from $M+\tau_M(\tau^*M)=(t-\theta)^{\min\{0,-n\}}\bF[\theta,t]$, we obtain:
\begin{equation}\label{eq:GM-for-carlitz}
G_{\underline{A}(n)}=\left[\frac{\KI\langle t\rangle}{\bF[\theta,t]}\stackrel{\Delta}{\longrightarrow} \frac{\KI\langle t\rangle}{(t-\theta)^{\min\{0,-n\}}\bF[\theta,t]}\right].
\end{equation}
The formula is thus deduced from Proposition \ref{prop:description-rB}.\\
Now assume $n>0$. The assertion on the rank follows from comparison with extensions of Hodge-Pink structures; more precisely, \cite[Thm 4.4]{gazda2} states that the rank of $\Ext^{1,\operatorname{reg}}_{A,\infty}(\bbone,\underline{A}(n))$ equals the dimension over $\KI=\bF(\!(1/t)\!)$ of a vector space of extensions of mixed Hodge-Pink structures:
\begin{equation}\label{eq:hdg-space}
\Ext^{1,\operatorname{ha}}_{\mathbf{HPk}^+_{\KI},\infty}(\mathbbm{1}^+,\underline{H}^+)
\end{equation}
(the reader is invited to refer to \emph{loc.\,cit.} for definitions and notations). Here, $\underline{H}^+$ is the Hodge-Pink structure attached to $\underline{A}(n)$; we are in the range of application of the quoted theorem as $\underline{A}(n)$ is pure of weight $-n<0$. From Thm. 3.28 in \emph{loc.\,cit.}, \eqref{eq:hdg-space} inserts an exact sequence of $\KI$-vector spaces:
\[
0\longrightarrow \underline{A}(n)^+_B\otimes_{\bF[t]}\KI\longrightarrow \frac{(t-\theta)^{-n}\KI[\![t-\theta]\!]}{\KI[\![t-\theta]\!]}\longrightarrow \Ext^{1,\operatorname{ha}}_{\mathbf{HPk}^+_{\KI},\infty}(\mathbbm{1}^+,\underline{H}^+)\longrightarrow 0
\]
from which the dimension of \eqref{eq:hdg-space} is easily computed. 
\end{proof}

We can already detect from Corollary \ref{cor:formula-for-carlitz} that the study of the integral $t$-motivic cohomology and that of the class module will require to work with the family of $\tau$\nobreakdash-difference equations 
\[
\Delta \xi=\xi-\frac{\xi^{(1)}}{(t-\theta)^n} \in (t-\theta)^{\min\{0,-n\}}\bF[\theta,t].
\]
In fact, most of what we will discuss next will simply require these formulas and no longer refer to \cite{gazda2}. In that sense, once one agrees with Corollary \ref{cor:formula-for-carlitz}, Sections \ref{sec:non-negative-twists} and \ref{sec:positive-twists} are mostly self contained.

\section{Non-positive twists}\label{sec:non-negative-twists}
Let $n\geq 0$ be an integer. In this section, we prove Theorems \ref{thm:polylogclass}, \ref{thm:mot-coh-car} and \ref{thm:rk-Cl} for $\underline{A}(-n)$. As computations are sensitive to the sign of the twist, we postpone the positive case to the next section.

\subsection{Twist zero}
For $n=0$, the twist $\underline{A}(n)$ equals the \emph{neutral $t$-motive} $\bbone$. In this case, we have the following.

\begin{Theorem}\label{thm:twist-zero}
The extension module $\Ext^{1,\operatorname{reg}}_{A,\infty}(\bbone,\bbone)$ and the class module $\operatorname{Cl}(\bbone)$ are both~trivial.
\end{Theorem}

\begin{proof}
The proof is covered from various stages of the next two subsections, so we simply give references. The result for the integral $t$-motivic cohomology is a special case of Theorem \ref{thm:revelo} below, whereas we get the result for the class module as a special case of Proposition \ref{prop:class-module-n-pos}.\footnote{One could obtain the result for the class module also as a special case of $n\leq 0$, but has to be aware that in this special case, the kernel determined in Lemma \ref{lem:kernel-beta} is also trivial, since $\bF[t]\cdot \nu_0=\bF[t]\subseteq \bF[\theta,t]$.}
\end{proof}

\subsection{Negative twists}
In this subsection, we consider the case of the $t$-motive $\underline{A}(-n)$ over $A$ for $n> 0$.
In the part on the integral $t$-motivic cohomology, however, we also allow $n=0$ to cover the case of twist zero.

\subsubsection*{Determination of integral $t$-motivic cohomology}

\begin{Theorem}\label{thm:revelo}
For $n\geq 0$, the module $\Ext^{1,\operatorname{reg}}_{A,\infty}(\bbone,\underline{A}(-n))$ is trivial.
\end{Theorem}

Before engaging in the proof of Theorem \ref{thm:revelo}, we set a few notations that we shall use again later on. First write $n=h(q-1)+\delta$ using Euclidean division of $n$ by $q-1$. In particular $0\leq \delta<q-1$. Next, we introduce the following \emph{flattened} version of $\omega^{-n}$:
\[
\nu_n:=(-\theta)^{-h}\prod_{j=0}^\infty \left(1-\frac{t}{\theta^{q^j}}\right)^{n}=\lambda_\theta^{\delta} \omega^{-n}.
\]
The element $\nu_n$ belongs to $\KI\langle t \rangle$ by choice of $\delta$ and is again invertible in $\TT$. \\

In virtue of Proposition \ref{prop:description-rB}, we have to determine the cohomology of:
\[
G_{\underline{A}(-n)}=\left[\frac{\KI\langle t \rangle}{\bF[\theta,t]}\stackrel{\Delta}{\longrightarrow} \frac{\KI\langle t \rangle}{\bF[\theta,t]}\right]
\]
(see \eqref{eq:GM-for-carlitz}). Observe that $\KI\langle t \rangle=\bF[\theta,t]\oplus \fm_\infty\langle t \rangle$, where we denote by $\fm_\infty\langle t \rangle$ the sub\nobreakdash-$\bF[t]$\nobreakdash-module of $\KI\langle t \rangle$ consisting of elements having Gauss norm less than $1$. Multiplying this decomposition by $\nu_n$ provides a new decomposition:
\begin{equation}\label{eq:decomposition-TateAlgebra}
\TT =  \bF[\theta,t]\cdot \nu_n\oplus \fm_\infty\cs{t}\cdot \nu_n.
\end{equation}
We also have:
\begin{equation}\label{eq:decomposition-TateAlgebra-ofnorm<1}
\fm_\infty\langle t \rangle =  \bF[\theta,t]_{\deg_\theta<h}\cdot \nu_n\oplus \fm_\infty\cs{t}\cdot \nu_n.
\end{equation}

We shall use the $\tau$-difference operator $\Delta=\Delta_{-n}$ introduced in \eqref{eq:Delta}. We begin with an easy observation:
\begin{Lemma}\label{lem:beta}
If $f\in \KI\cs{t}$, then $\Delta(f\cdot \nu_n)=(f-(-\theta)^{\delta}f^{(1)})\cdot \nu_n$. In particular, $\Delta$~preserves the decomposition \eqref{eq:decomposition-TateAlgebra} and acts as an isomorphism on the summand \mbox{$\fm_\infty \cs{t}\cdot \nu_n$}. 
\end{Lemma}

The main ingredient in the proof of Theorem \ref{thm:revelo} is contained in the following:
\begin{Lemma}\label{lem:kernel-beta}
$\operatorname{H}^0(G_{\underline{A}(-n)})$ is trivial if $q-1\nmid n$ and is $(\bF[t]\cdot \nu_n +\bF[\theta,t])/\bF[\theta,t]$ otherwise. 
\end{Lemma}

\begin{proof}
Let $\xi\in \KI\langle t \rangle$ be the lift of a non zero element in the kernel. Up to subtracting a certain element of $\bF[\theta,t]$ from $\xi$, we may and will assume that $\xi\in \fm_\infty\langle t \rangle$. By mean of the decomposition \eqref{eq:decomposition-TateAlgebra-ofnorm<1}, we write $\xi=P(t)\nu_n+g(t)\nu_n$ for $P(t)\in \bF[\theta,t]_{\deg_\theta < h}$ and $g(t)\in \fm_\infty\cs{t}$. After dividing $\xi$ by a suitable power of $t$, we can further assume that $P(0)\ne 0$ or $g(0)\ne 0$.

By assumption, there exists $a(t)\in \bF[\theta,t]$ such that
\[ \Delta\left(P(t)\nu+g(t)\nu\right)=a(t)\in \bF[\theta,t]. \]
Using Lemma \ref{lem:beta} and evaluating the equation at $t=0$ leads to:
\[ (P(0)-P(0)^q\cdot (-\theta)^{\delta})+(g(0)-g(0)^q\cdot (-\theta)^{\delta}) = a(0)\cdot (-\theta)^h\in \bF[\theta]. \]
As $(P(0)-P(0)^q\cdot (-\theta)^{\delta})\in \bF[\theta]$ and $(g(0)-g(0)^q\cdot (-\theta)^{\delta})\in \fm_\infty$, this implies:
\begin{equation}\label{eq:eq-for-g}
    g(0)-g(0)^q\cdot (-\theta)^{\delta}=0
\end{equation} 
and
\begin{equation}\label{eq:eq-for-P}
P(0)-P(0)^q\cdot (-\theta)^{\delta}=a(0)\cdot (-\theta)^h.
\end{equation}
From equation \eqref{eq:eq-for-g}, we obtain
$\abs{g(0)}=q^{-\frac{\delta}{q-1}}>q^{-1}$ a contradiction, unless $g(0)=0$.

For $P(0)$ -- which we now can assume to be non-zero -- we consider two cases:\\
If $q-1\nmid n$, i.e. $\delta>0$, the vanishing order at $\theta=0$ of $P(0)^q\cdot (-\theta)^{\delta}$ is strictly bigger than that of $P(0)$, and hence by Equation \eqref{eq:eq-for-P}, the vanishing order at $\theta=0$ of $P(0)$ has to be at least $h$, contradicting that $P(0)$ has degree less than $h$. Hence the map is injective. \\
If $q-1|n$, i.e. $\delta=0$, then we can assume that $P(0)$ vanishes at $\theta=0$ up to subtracting from $\xi$ a suitable $\bF[t]$-multiple of $\nu_n=\omega^{-n}$. Yet if the difference is non zero, then we get a contradiction in the same manner as in the case $\delta>0$.
\end{proof}

\begin{proof}[Proof of Theorem \ref{thm:revelo}]
Let $\xi\in \KI\langle t \rangle$ be such that $\xi-(t-\theta)^n \xi^{(1)}\in \bF[\theta,t]$. By Lemma \ref{lem:kernel-beta}, $\xi$ belongs to $\bF[\theta,t]+\mathbf{1}_{q-1|n}\bF[t]\omega(t)^{-n}$. Therefore, Corollary \ref{cor:formula-for-carlitz} implies that $\Ext^{1,\operatorname{reg}}_{A,\infty}(\bbone,\underline{A}(-n))=(0)$.
\end{proof}

\subsubsection*{Determination of the rank of the class module}
We now turn to the determination of the rank of the class module $\operatorname{Cl}(\underline{A}(-n))$ for $n>0$.
\begin{Theorem}\label{thm:class-module-n-neg}
For $n>0$, the rank of $\operatorname{Cl}(\underline{A}(-n))$ is $0$ if $q-1\nmid n$ and is $1$ if $q-1|n$.
\end{Theorem}

We denote by $\gamma$ the $\bF[t]$-linear map given as the composition of:
\begin{equation}\label{eq:gamma}
\gamma:\bigoplus_{k=1}^{h} \bF[t]\cdot \theta^{-k} \hookrightarrow \KI\langle t \rangle \stackrel{\Delta}{\longrightarrow} \KI\langle t \rangle \twoheadrightarrow \bigoplus_{k=1}^{h}  \bF[t]\cdot \theta^{-k},
\end{equation}
where the projection map is taken orthogonally to $\mI\langle t \rangle\cdot \nu_n$. We have the following:
\begin{Lemma}\label{lem:quasi-isomorphism}
The complex $\left[\displaystyle\bigoplus_{k=1}^{h} \bF[t]\cdot \theta^{-k}\stackrel{\gamma}{\longrightarrow} \bigoplus_{k=1}^{h} \bF[t]\cdot \theta^{-k}\right]$ sitting in degrees $\{0,1\}$ is quasi-isomorphic to $G_{\underline{A}(-n)}$.
\end{Lemma}

\begin{proof}
Let $\displaystyle\pi:\KI\langle t \rangle/\bF[\theta,t]\to \bigoplus_{k=1}^{h} \bF[t]\cdot \theta^{-k}$ be the composition of 
\[
\pi:\frac{\KI\langle t \rangle}{\bF[\theta,t]} \stackrel{\sim}{\longleftarrow} \fm_\infty\langle t \rangle = \bigoplus_{k=1}^{h} \bF[t]\cdot \theta^{-k} \oplus \fm_\infty\cs{t}\cdot \nu_n \twoheadrightarrow  \bigoplus_{k=1}^{h} \bF[t]\cdot \theta^{-k}.
\]
The diagram 
\begin{equation}
\begin{tikzcd}
\displaystyle\frac{\KI\langle t \rangle}{\bF[\theta,t]} \arrow[r,"\Delta"] \arrow[d,"\pi"] & \displaystyle\frac{\KI\langle t \rangle}{\bF[\theta,t]} \arrow[d,"\pi"] \\
\displaystyle\bigoplus_{k=1}^{h} \bF[t]\cdot \theta^{-k}  \arrow[r,"\gamma"] & \displaystyle\bigoplus_{k=1}^{h} \bF[t]\cdot \theta^{-k}
\end{tikzcd} \nonumber
\end{equation}
is commutative, and as $\Delta$ is an isomorphism on the summand $\fm_\infty\langle t\rangle \cdot \nu_n$ (from Lemma \ref{lem:beta}), it induces the sought quasi-isomorphism. 
\end{proof}

\begin{proof}[Proof of Theorem \ref{thm:class-module-n-neg}]
By Proposition \ref{prop:description-rB}, $\operatorname{Cl}(\underline{A}(-n))$ is isomorphic to $\operatorname{H}^1(G_{\underline{A}(-n)})$. By Lemma \ref{lem:quasi-isomorphism}, $G_{\underline{A}(-n)}$ is a perfect complex of Euler characteristic zero. Hence, Lemma \ref{lem:kernel-beta} concludes the proof. 
\end{proof}

\subsection{Relation with Carlitz zeta values}
In this paragraph, we prove Theorem \ref{thm:zeta?}. The Carlitz zeta value at $-n$, denoted $\zeta_C(-n)$, is defined as the sum 
\[
\zeta_C(-n)=\sum_{d\geq 0}{S_d(n)}, \quad \text{where} \quad S_d(n):=\sum_{\substack{a~\text{monic}\\ \deg(a)=d}}{a(t)^{n}}.
\]
For $d$ large enough (depending on $n$) the summand $S_d(-n)$ vanishes, from which we deduce $\zeta_C(-n)\in \bF[t]$ \cite[\S 4]{goss-zeta}. The combinatorics of the sums $S_d(n)$ was studied in details by Thakur in \cite{thakur-powersums}. Consider rather the sum:
\[
Z(x,-n):=\sum_{d\geq 0}{x^d S_d(n)}
\]
which is a polynomial in both $x$ and $t$. We have $Z(1,-n)=\zeta_C(-n)$ by design but, more generally, let $\zeta_C^*(-n)\in \bF[t]$ be the first non zero coefficient in the expansion of $Z(x,-n)$ at $x=1$. Using the trace formula as expressed in Taelman's \cite{taelman-woodshole}, we prove:

\begin{Theorem}\label{thm:zeta}
We have $\operatorname{Fitt} \operatorname{Cl}(\underline{A}(-n))=(\zeta_C^*(-n))$.
\end{Theorem}
\begin{proof}
Consider the scheme $Y:=\Spec \bF[t]\times \bP^1_{\bF}$ with the projection map $\operatorname{pr}:Y\to \bP^1_{\bF}$ where $\bP^1_{\bF}$ is seen as the curve associated to the function field $\bF(\theta)$. Given an integer $d$, we define the coherent $\cO_Y$-module $\cO(-d\infty)$ as $\operatorname{pr}^*\cO_{\bP^1}(-d\infty)$. If $d$ satisfies $d>\frac{n}{q-1}$, there is a morphism of coherent $\cO_Y$-modules:
\[
\delta_n:\tau^* \cO(-d\infty) \longrightarrow \cO(-(d+1)\infty)
\]
given on sections by mapping $\tau^* f$ to $(t-\theta)^n \tau(f)$, where $\tau$ denotes the endomorphism $\id\times \Frob_q$ of $Y$. By \cite[Prop. 8.5]{taelman-woodshole}, we have 
\begin{equation}\label{eq:woodshole}
Z(x,-n)=\prod_{i\geq 0}{\operatorname{det}_{\bF[t]}\left(\id-x\delta_n|\operatorname{H}^i(Y,\cO(-d\infty))\right)^{(-1)^{i+1}}}
\end{equation}
where we denoted $\delta_n$ the induced morphism on cohomology modules. Now, we have $\operatorname{H}^i(Y,\cO(-d\infty))\cong \bF[t]\otimes \operatorname{H}^i(\bP^1_\bF,\cO_{\bP^1}(-d\infty))$ by base-change. Using the classical cover of $\bP^1_{\bF}$ by the two affine subschemes $\Spec \bF[\theta]$, $\Spec \bF[\theta^{-1}]$ glued along $\Spec \bF[\theta, \theta^{-1}]$, we compute the cohomology of $\cO_{\bP^1}(-d\infty)$ by the complex in degrees in $\{0,1\}$:
\[
\left[\bF[\theta]\oplus \theta^{-d}\bF[\theta^{-1}]\longrightarrow \bF[\theta,\theta^{-1}]\right]
\]
where the boundary map is the subtraction. That is, $\operatorname{H}^i(\bP^1_\bF,\cO_{\bP^1}(-d\infty))=(0)$ unless $i=1$, where:
\[
\operatorname{H}^1(\bP^1_\bF,\cO_{\bP^1}(-d\infty))=\frac{\bF[\theta,\theta^{-1}]}{\bF[\theta]+\theta^{-d}\bF[\theta^{-1}]}.
\]
From this well-known computation, we deduce with \eqref{eq:woodshole} that 
\begin{align*}
Z(x,-n)&=\operatorname{det}_{\bF[t]}\left(\id-x(\id-\Delta) \bigg| \frac{\bF[\theta,\theta^{-1},t]}{\bF[\theta,t]+\theta^{-d}\bF[\theta^{-1},t]}\right) \\
&=\operatorname{det}_{\bF[t]}\left(\id-x(\id-\gamma) \bigg| \bigoplus_{k=1}^{d-1}{\bF[t]\cdot \theta^{-k}}\right).
\end{align*}
Now, as $(q-1)(h+1)>n$ by choice of $h$, we can take $d=h+1$. The result is then deduced from Lemma \ref{lem:quasi-isomorphism} and generalities on Fitting ideals.
\end{proof}

\begin{Remark}
We also deduce from \eqref{eq:woodshole} that $\operatorname{rk}\operatorname{Cl}(\underline{A}(-n))=\operatorname{ord}_{x=1}Z(x,-n)$. Combined with Theorem \ref{thm:rk-Cl}, we recover a result of Goss stating that $\zeta_C(-n)=0$ if, and only if, $q-1|n$. This is reminiscent of the classical Riemann zeta values vanishing at even negative integers.
\end{Remark}

\begin{Remark}\label{rmj:elementary-strategy}
We believe that Theorem \ref{thm:zeta} can be proven by elementary means (i.e. without the use of Taelman's trace formula). Considering the action of $\gamma$ on the basis $((-\theta)^{-1},...,(-\theta)^{-h})$ seen in $\KI\langle t \rangle/\fm_{\infty}\langle t \rangle\cdot \nu_n$, a more combinatorial formula for the polynomial $P_n(x):=\operatorname{det}_{\bF[t]}(\id-x(\id-\gamma))$ is revealed:
\[
P_n(x)=\operatorname{det}_{\bF[t]}(\id_h-xM_n),~\text{where}~M_n=\left( \binom{n}{iq-j+\delta}\cdot t^{iq-j+\delta} \right)_{0\leq i,j< h}\in \operatorname{Mat}_h(\bF_{\!p}[t]).
\]
On the other-hand, following Goss' recursion formula \cite[Thm. 5.6]{goss-zeta}, one shows:
\begin{equation}\label{eq:recursion}
Z(x,-n)=1-x\displaystyle \!\sum_{0\leq i< h }{\!\binom{n}{i(q-1)+\delta} t^{i(q-1)+\delta} Z(x,-i(q-1)-\delta)}.
\end{equation}
The formula \eqref{eq:recursion} is strongly suggestive regarding the form of $M_n$, and Theorem \ref{thm:zeta} would follow if one was able to show that $P_n(x)$ satisfies the recursion formula \eqref{eq:recursion}. Surprisingly, all our efforts in proving this elementary-looking problem were vain; classical determinantal identities seem not enough to answer it, and sharp congruences for binomial coefficients mod $p$ -- which still elude us -- might be at play behind this result.
\end{Remark}

The following consequence was already stated in the introduction.
\begin{Corollary}\label{cor:thakur}
Assume that $q=p$ is prime and let $n\geq 0$. Then, the map 
\begin{equation}\label{eq:rBzeta}
r_B:\Ext^{1,\operatorname{reg}}_A(\mathbbm{1},\underline{A}(-n))\longrightarrow \operatorname{H}^1(G_\infty,\underline{A}(-n)_B),
\end{equation}
is an isomorphism if, and only if, the sum of the digits of $n$ written in base $p$ is $<p-1$. 
\end{Corollary}
\begin{proof}
From Theorem \ref{thm:revelo}, the morphism \eqref{eq:rBzeta} is injective. From Theorems \ref{thm:class-module-n-neg} and \ref{thm:zeta}, it is an isomorphism if, and only if, $q-1\nmid n$ and $\zeta_C(-n)\in \bF^\times$. Now, for $q-1\nmid n$, we claim that the following points are equivalent:
\begin{enumerate}[label=(\Roman*)]
    \item\label{item:sumdigits} The sum of digit of $n$ in base $p$ is $<p-1$,
    \item\label{item:Sd(n)vanishes} For all $d>0$, $S_d(n)=0$,
    \item\label{item:Zconstant} $Z(x,-n)\in \bF$,
    \item\label{item:ZetaSmall} $\zeta(-n)\in \bF$.
    \item\label{item:Zeta1} $\zeta(-n)=1$.
\end{enumerate}
The equivalence between \ref{item:sumdigits} and \ref{item:Sd(n)vanishes} is due to Lee \cite[Thms.\,2 \& 3]{lee}. The equivalence between \ref{item:Sd(n)vanishes} and \ref{item:Zconstant} follows from the definition of $Z(x,-n)$. Now, clearly $Z(x,-n)\in \bF$ implies $\zeta(-n)\in \bF$. The converse is deduced from Lemma \ref{lem:thakur} below, proving the equivalence between \ref{item:Zconstant} and \ref{item:ZetaSmall}. That \ref{item:ZetaSmall} is equivalent to \ref{item:Zeta1} follows from $S_0(n)=1$ using the equivalence between \ref{item:Zconstant} and \ref{item:ZetaSmall}. 
\end{proof}

We thank Dinesh Thakur for sharing with us the following result.
\begin{Lemma}\label{lem:thakur}
Assume $q-1\nmid n$ and let $\ell_n$ be the highest integer $\ell$ such that $S_{\ell}(n)\neq 0$; i.e. the degree in $x$ of $Z(x,-n)$. Then, the degrees of the family $(S_d(n))_{d=1,...,\ell_n}$ are strictly increasing. 
\end{Lemma}
\begin{proof}
Let $s_d(n)$ be\footnote{That is, $-s_d(-n)$ according to the notations of \cite{thakur-powersums}.} the degree of the polynomial $S_d(n)$. By \cite[(2)]{thakur-powersums}, when $0<d<\ell_n$ we have 
\[
s_d(n)=s_1^{(d-1)}(n)+\ldots+s_1^{(2)}(n)+s_1(n)
\]
where $s_1^{(i)}(n):=s_1(s_1(\cdots (s_1(n))\cdots ))$ $i$-times. To prove that $s_d(n)$ is strictly increasing, it suffices to show that $s_1^{(i)}(n)$ is non zero. Yet, from Thm.\,14 in \emph{loc.\,cit.}, we have $s_1(k)\equiv k\pmod{q-1}$ from which we deduce that $s_1^{(i)}(n)=0$ can only happen if $n\equiv 0\pmod{q-1}$.
\end{proof}

\section{Positive twists}\label{sec:positive-twists}
For this section, we are concerned with the $t$-motive $\underline{A}(n)$ for a positive number $n>0$. We begin this time with the simpler case of the class module. Surprisingly, the methods quite differ from the $\underline{A}(-n)$-case.
In the determination of the class module, we also allow $n=0$ to cover the case of twist zero.

\subsection{The $\xi$ functions}
Recall that $K$ is short for the field $\bF(\theta)$. We denote by $\overline{K}$ the algebraic closure of $K$ inside $\CI$ and we let $\cO_{\overline{K}}$ be the integral closure of $\bF[\theta]$ in $\overline{K}$. In our study of positive twists, a crucial role is played by solutions $\xi$ in the Tate algebra of the equation 
\[(t-\theta)^n\Delta \xi=(t-\theta)^n\xi - \xi^{(1)} = e\] for various $e$ in $\CI\langle t\rangle$. Whenever $e$ satisfies $\|e\|<q^{\frac{nq}{q-1}}$, a particular solution of this equation is given by the convergent series in $\CI\langle t \rangle$:
\begin{equation}\label{eq:xi}
\xi_{e}(t):= \frac{e}{(t-\theta)^n}+\sum_{k=1}^\infty \frac{e^{(k)}}{(t-\theta)^n(t-\theta^q)^n\cdots (t-\theta^{q^k})^n},
\end{equation}
and all the solutions are provided by the affine set $\xi_{e}(t)+\bF[t] \omega(t)^n$. 

A key property of the series $\xi_\alpha$, for some $\alpha\in \CI$ of norm $\|\alpha\|<q^{\frac{nq}{q-1}}$, is its relation with the Carlitz $n$th polylogarithm:
\begin{equation}\label{eq:xi-and-polylog}
(t-\theta)^n\xi_{\alpha}(t)|_{t=\theta}=\Li_{n}(\alpha).
\end{equation}
We will also frequently use that the mapping $e\mapsto \xi_e$ is $\bF[t]$-linear.

\subsubsection*{Linear relations}
As we shall see below, linear relations among the elements $\xi_e(t)$, $\omega(t)^n$ and polynomials in $\overline{K}[t]$, for various polynomials $e$, dictate the structure of the module of extensions of Carlitz twists. In this subsection, we establish the linear relations among those functions and will draw consequences on extensions in the following sections. 

\begin{Proposition}\label{prop:linear-relation-among-xi}
Let $\cR\subseteq\cO_{\overline{K}}$ be the sub-algebra $\overline{\bF}[\theta^r\mid r\in \bQ_+]$.
\begin{enumerate}
\item\label{item:xi_e-not-rational} For $0\ne e\in \cO_{\overline{K}}[t]$ with norm $\norm{e}<q^n$, one has $\xi_e\not\in \overline{K}(t)$.
\item\label{item:reduce-norm} For each $e'\in \cR[t]$ with norm $\norm{e'}<q^{\frac{nq}{q-1}}$, there exists a unique $e\in \cR[t]$ with norm $\norm{e}<q^n$ such that
$\xi_{e'}-\xi_{e}\in \overline{K}(t)$.
For that $e$, one even has $\xi_{e'}-\xi_{e}\in \cR[t]$.
\item\label{item:e} There exists a unique $e_n\in \cO_{\overline{K}}[t]$ with norm $\|e_n\|<q^n$, such that $\xi_{e_n}+\omega^n\in \overline{K}(t)$. One even has $\xi_{e_n}+\omega^n\in \cO_{\overline{K}}[t]$.
\item\label{item:e-with-b} For each $b\in \bF[t]$, and $e_n$ as in the previous statement, the element $e'=b\cdot e_n$ is the unique element in $\cO_{\overline{K}}[t]$ with norm $<q^n$, such that $\xi_{e'}+b\cdot \omega^n\in \overline{K}(t)$.
\end{enumerate}
\end{Proposition}

Here is the recipe to construct the element $e_n$ as in point \ref{item:e}. First, let $g_n(t)\in \cO_{\overline{K}}[t]$ be the \emph{best integral approximation to $\omega(t)^n$}, that is, the unique polynomial such that $\|g_n(t)-\omega(t)^n\|<1$. Then $e_n$ is defined as 
\begin{equation}\label{eq:construction-e}
e_n(t):=(t-\theta)^n g_n(t)-g_n(t)^{(1)}\in \cO_{\overline{K}}[t]\setminus\{0\}.
\end{equation}
\begin{Remark}\label{rem:p-power}
From the uniqueness of $e_n$, we automatically get $e_{np}=e_n^p$, which can also be deduced from formula \eqref{eq:construction-e}.
\end{Remark}

\begin{proof}[Proof of Proposition \ref{prop:linear-relation-among-xi}]
We prove \ref{item:xi_e-not-rational}. Assume that $\xi_e\in \overline{K}(t)$. As $\xi_e$ lies in the Tate algebra, its evaluation at $t=a$ for $a\in \overline{K}$ of norm $\leq 1$ is given by a convergent series; that is, all possible poles of $\xi_e$ have norm $>1$. If it has a pole at $t=a$, $a\in \overline{K}$, then $\xi_e^{(1)}$ has a pole at $a^q$. But as $(t-\theta)^n \xi_e = \xi_e^{(1)} - e$, $\xi_e$ has a pole at $t=a^q$. Inductively, we obtain that it has infinitely many poles, namely at $a$, $a^q$, $a^{q^2}$ etc, of increasing norm, hence distinct. This contradicts the assumption of being a rational function, so $\xi_e$ does not have poles, i.e.~$\xi_e\in \overline{K}[t]$. \\
Therefore, we can write $\xi_e =a\cdot \prod_{i=1}^l (t-a_i)$ for some $a,a_i\in \overline{K}$, and we have $\|a\|<1$ as $\|\xi_e\|=\|e\|\cdot q^{-n}<1$. In particular, $a\not\in \cO_{\overline{K}}$. But the equation $(t-\theta)^n \xi_e -\xi^{(1)}_e=e$ implies that the leading coefficient of $e$ as a polynomial in $t$ is $a$, contradicting $e \in \cO_{\overline{K}}[t]$. 

Next we prove \ref{item:reduce-norm}. Uniqueness is clear from point \ref{item:xi_e-not-rational}, since the difference $e-\tilde{e}$ of two candidates would satisfy $\xi_{e-\tilde{e}}=\xi_e-\xi_{\tilde{e}}\in \overline{K}(t)$. We show existence by giving an algorithm to find $e$. The key ingredient is the relation
\begin{equation} \label{eq:polynomial-xi}
\xi_{(t-\theta)^n d-d^{(1)}}=\xi_{(t-\theta)^nd}-\xi_{d^{(1)}}=d
\end{equation}
for all $d\in \cO_{\overline{K}}[t]$ with norm $\norm{d}<q^{\frac{n}{q-1}}$. We leave the verification of this relation to the reader.

If $\norm{e'}<q^n$, there is nothing to prove. So we assume $\norm{e'}\geq q^n$. As $e\in \cR$, there exists an integer $\delta>0$ for which $e'\in \bar{\bF}[\theta^{1/\delta},t]$. The polynomial division algorithm of $e'$ by $(t-\theta)^n$ as polynomials in the variable $\theta^{1/\delta}$ provides $d_1,c_1\in \bar{\bF}[t][\theta^{1/\delta}]\subset \cR[t]$ such that $e'=(t-\theta)^n d_1+c_1$ with $\norm{c_1}<q^n$.
By equation \eqref{eq:polynomial-xi}, the element $e_1=d_1^{(1)}+c_1$ satisfies
\[ \xi_{e'}-\xi_{e_1}=\xi_{(t-\theta)^nd_1-d_1^{(1)}}=d_1\in \cR[t], \]
since $\norm{d_1}=\norm{e'}\cdot q^{-n}<q^{\frac{n}{q-1}}$.
If $\norm{e_1}<q^n$, this is the desired $e$. Otherwise, we repeat this process to $e_1$, to obtain inductively a family of elements $(e_k)_{k\geq 1}$ in  $\cR[t]$. We claim that there is some $m\geq 1$ for which $\norm{e_m}<q^n$; this implies the result, as then
\[ \xi_{e'}-\xi_{e_m}=d_1+d_2+\ldots+d_m \in \cR[t],\]
where the $d_i$ are the corresponding elements in the process. It remains to prove the claim. If at a step, one has $e_j=d_j^{(1)}+c_j$ with $\norm{e_j}\geq q^n$, then 
\[ \norm{e_j}=\norm{d_j}^q=\left(\norm{e_{j-1}}\cdot q^{-n}\right)^q.\]
Hence,
\[
\frac{\,\,\norm{e_j}\,\,}{q^{\frac{nq}{q-1}}} = \frac{\left(\norm{e_{j-1}}\cdot q^{-n}\right)^q}{q^{\frac{nq}{q-1}}} 
=\frac{\norm{e_{j-1}}^q}{q^{nq}\cdot q^{\frac{nq}{q-1}}} =\left( \frac{\,\,\norm{e_{j-1}}\,\,}{q^{\frac{nq}{q-1}}}\right)^q.
\]
Therefore, since $\norm{e'}<q^{\frac{nq}{q-1}}$, the sequence of quotients  $\frac{\,\,\norm{e_j}\,\,}{q^{\frac{nq}{q-1}}}$ will go below $\frac{q^n}{q^{\frac{nq}{q-1}}}$, hence $(e_k)_{k}$ reaches an element $e_m$ with $\norm{e_m}<q^n$.

We turn to \ref{item:e}. Uniqueness is again clear from point \ref{item:xi_e-not-rational}. So we have to show existence; more precisely, we show that the polynomial $e=e_n$ constructed in \eqref{eq:construction-e} is as required. Using the functional equation of $\omega^n$, we have
\[ e=(t-\theta)^n(g_n-\omega^n)-(g_n-\omega^n)^{(1)}, \]
and hence
\[ \|e\|\leq \max\{ \norm{(t-\theta)^n(g_n-\omega^n)}, \norm{(g_n-\omega^n)^{(1)}} \} < q^n. \]
But since $(t-\theta)^n\xi_e-\xi_e^{(1)}=e$, we get
\[ (t-\theta)^n(\xi_e-g_n+\omega^n)=(\xi_e-g_n+\omega^n)^{(1)} \]
and hence $\xi_e-g_n+\omega^n\in \bF[t]\cdot \omega^n$. Yet $\norm{\xi_e-g_n+\omega^n}<1<\norm{\omega^n}$ which is only possible if $\xi_e-g_n+\omega^n=0$, and which amounts to $\xi_e+\omega^n=g_n\in \cO_{\overline{K}}[t]$. 

It remains to prove \ref{item:e-with-b}. Again, uniqueness is given by point \ref{item:xi_e-not-rational}. Similarly, it is easily verified that $e'=b\cdot e_n$ satisfies the given conditions.
\end{proof}

In order to state at best the different linear relations, we give a name to the module measuring those: for a family of elements $\underline{\alpha}=(\alpha_1,\ldots,\alpha_s)$ in $\cO_{\overline{K}}$ having norm $\|\alpha_i\|<q^{\frac{nq}{q-1}}$, we consider the $\overline{K}[t]$-module
\begin{equation}\label{eq:Z}
\cZ(\underline{\alpha}):=\{(z_1,\ldots,z_s,y)\in \overline{K}[t]^{s+1}|z_1(t) \xi_{\alpha_1}(t)+\ldots+z_s(t)\xi_{\alpha_s}(t)+y(t)\omega(t)^n\in \overline{K}(t)\}.
\end{equation}

Some facts about its structure is given by the next result:
\begin{Proposition}\label{prop:invariant-generators}
The $\overline{K}[t]$-module $\cZ(\underline{\alpha})$ is finite free and $\cZ_{\overline{K}(t)}:=\cZ(\underline{\alpha})\otimes_{\overline{K}[t]}\overline{K}(t)$ is generated by elements in $\bF[t]^{s+1}$.
\end{Proposition}
\begin{proof}
That $\cZ(\underline{\alpha})$ is finite free follows from the fact that it is a submodule of a finite free module over $\overline{K}[t]$. For the second part of the statement, observe that from the $\tau$-difference equations satisfied by $\omega^n$ and $\xi_{\alpha_i}$, one deduces that the vector space $\cZ_{\overline{K}(t)}$ is stable under taking twists, hence is a difference submodule of $\overline{K}(t)^{s+1}$. As $\overline{K}(t)^{s+1}$ is generated by $\tau$-invariant elements, so is the difference submodule $\cZ_{\overline{K}(t)}$ (see \cite[Thm. 1.32]{vdps-difference}, or \cite[Prop.~3.5]{maurischat-pv-theory} for a detailed proof). That is, $\cZ_{\overline{K}(t)}$ is generated by elements in 
$\bF(t)^{s+1}\cap \cZ_{\overline{K}(t)}$. By clearing denominators, we also get a generating set of elements in $\bF[t]^{s+1}$.
\end{proof}

In order to explicit generators of $\cZ(\underline{\alpha})$, we have the next:
\begin{Proposition}\label{prop:Z-alpha}
Suppose that the family $\underline{\alpha}=(\alpha_1,\ldots,\alpha_s)$ in $\cO_{\overline{K}}$ is linearly independent over $\bF$ with norms $\|\alpha_i\|<q^n$. If $e_n$ of Prop.~\ref{prop:linear-relation-among-xi} can be written as $\sum_{i=1}^s c_i\alpha_i$ for some coefficients $c_i\in \bF[t]$, then $\cZ(\underline{\alpha})=\overline{K}[t]\cdot (c_1,\ldots,c_s,1)$. Otherwise, $\cZ(\underline{\alpha})=\{0\}$.
\end{Proposition}

\begin{proof}
Consider a vector $(z_1,\ldots, z_s,y)\in \cZ(\underline{\alpha})_{\overline{K}(t)}$ invariant under twist; i.e. with coordinates in $\bF(t)$. After multiplying with the least common multiple of their denominators, we can assume $z_i\in \bF[t]$ for all $i\in \{1,\ldots,s\}$, as well as $y\in \bF[t]$. Let $e':=\sum_{i=1}^s z_i\alpha_i$. By $\bF[t]$-linearity of the assignment $e\mapsto \xi_e$, we have
$\xi_{e'}=\sum_{i=1}^s z_i\xi_{\alpha_i}$, hence
\[  \xi_{e'} +y\omega^n \in \overline{K}(t). \]
By Proposition \ref{prop:linear-relation-among-xi}, we obtain $e'=\tilde{b}e_n$ for some $\tilde{b}\in \bF[t]$. 

If $e_n=\sum_{i=1}^s c_i\alpha_i$, we have $z_i=y c_i$ for all $i$, as the family $(\alpha_1,\ldots,\alpha_s)$ is linearly independent. Hence by Proposition \ref{prop:invariant-generators}, 
\[ \cZ(\underline{\alpha})_{\overline{K}(t)}=\overline{K}(t)\cdot (c_1,\ldots,c_s,1),\quad \text{and} \quad \cZ(\underline{\alpha})= \cZ(\underline{\alpha})_{\overline{K}(t)}\cap \overline{K}[t]^{s+1}=\overline{K}[t]\cdot (c_1,\ldots,c_s,1).\]

On the opposite, if $e_n$ is not in the $\bF[t]$-span of the $\alpha_i$, one has $\tilde{b}=0=\tilde{a}_i$; that is $\cZ(\underline{\alpha})_{\overline{K}(t)}=\cZ(\underline{\alpha})=0$.
\end{proof}

We record the following Lemma in the case the elements $(\alpha_1,\cdots,\alpha_s)$ are in $\bF[\theta]$.
\begin{Lemma}\label{lem:equiv-cond-for-tuple}
Let $\alpha_1,\ldots,\alpha_s$ be elements in $\bF[\theta]$ of degree $<\frac{qn}{q-1}$. Given a tuple $(c_1,\ldots, c_s,d)\in \bF[t]^{s+1}$, one has
\[ (c_1,\ldots, c_s,d)\in \cZ(\underline{\alpha}) \Longleftrightarrow \sum_{i=1}^s c_i\xi_{\alpha_i}+d\omega^n\in \bF[\theta,t]. \]
\end{Lemma}

\begin{proof}
The implication ``$\Longleftarrow$'' follows by definition of $\cZ(\underline{\alpha})$. We show the converse one: suppose $(c_1,\ldots, c_s,d)\in \cZ(\underline{\alpha})$ and set
$e':=\sum_{i=1}^s c_i\alpha_i\in \bF[\theta,t]$. If $\norm{e'}<q^n$, the implication holds by Proposition \ref{prop:linear-independence}. So we assume $\norm{e'}\geq q^n$. Let $e\in \bF[\theta,t]$ be the unique element with $\norm{e}<q^n$ and $\xi_{e'}-\xi_e\in \bF[\theta,t]$ given by Proposition \ref{prop:linear-independence}\eqref{item:reduce-norm}; we have $\xi_e+d\omega^n\in \overline{K}(t)$ and thus $e=d\cdot e_n$ using Proposition \ref{prop:linear-independence} again. So, either $d=0$, in which case $e=0$ and
\[ \sum_{i=1}^s c_i \xi_{\alpha_i}+d\omega^n=\xi_{e'-e}\in \bF[\theta,t], \]
or $d\ne 0$ and $e_n=\frac{e}{d}\in \bF[\theta](t)\cap \cO_{\overline{K}}[t]=\bF[\theta,t]$.
By Lemma \ref{lem:en-and-gn} below, we also have $g_n\in \bF[\theta,t]$. This amounts to
\[ \sum_{i=1}^s c_i \xi_{\alpha_i}+d\omega^n=\xi_{e'-e}+\xi_e+d\omega^n=\xi_{e'-e}+d g_n \in \bF[\theta,t] \]
as desired.
\end{proof}

\subsubsection*{Properties of the polynomial $e_n(t)$}
We pursue with some key properties of the polynomials $e_n(t)$, constructed in \eqref{eq:construction-e}, which will be used later on to compute the torsion in the integral $t$-motivic cohomology.
\begin{Lemma}\label{lem:en-and-gn}
For the element $e_n$ given in \eqref{eq:construction-e}, we have
\[ g_n\in \bF[\theta,t] \Longleftrightarrow e_n\in \bF[\theta,t] \Longleftrightarrow (q-1)\mid n.\]
\end{Lemma}

\begin{proof}
This is immediate once the following observation is made: the coefficients of $\omega^n$ consist solely of powers $(-\theta)^m$ where $m\in \frac{n}{q-1}+\bZ$. So the same holds for $g_n$ and $e_n$. 
\end{proof}

\begin{Lemma}\label{lem:equivalence-of-vanishing}
Suppose that $(q-1)|n$ and let $\zeta\in \bar{\bF}$ be algebraic of degree $r\geq 1$ over $\bF$. Then, the following are equivalent:
\begin{enumerate}[label=$(\roman*)$]
    \item\label{i:zero-of-e} $e_n(\zeta)=0$,
    \item\label{ii:zero-of-rho} $\xi_{e_n}(\zeta)=0$,
    \item\label{iii:omegan-polynomial} $\omega(\zeta)^n\in \bar{\bF}[\theta]$,
    \item\label{item:iv-qr-1dividesn} $q^r-1|n$.
\end{enumerate}
\end{Lemma}
\begin{proof}
We first prove the equivalence between assertions \ref{i:zero-of-e} and \ref{ii:zero-of-rho}. We have:
\[
e_n(\zeta)=(\zeta-\theta)^n\xi_{e_n}(\zeta)-\xi_{e_n}^{(1)}(\zeta)
\]
so that $e_n(\zeta)=0$ if and only if $\xi_{e_n}^{(1)}(\zeta)=(\zeta-\theta)^n \xi_{e_n}(\zeta)$. As $\xi_{e_n}(t)=g_n(t)-\omega(t)^n$ (Proposition \ref{prop:linear-relation-among-xi}), the norm of $\xi_{e_n}(\zeta)$ is $<1$ by definition of $g_n$. Therefore, the inequalities $|\xi_{e_n}^{(1)}(\zeta)|\leq |\xi_{e_n}(\zeta)|^q\leq |(\zeta-\theta)\xi_{e_n}(\zeta)|$ are equalities if and only if $\xi_{e_n}(\zeta)=0$.

We turn to the equivalence between \ref{ii:zero-of-rho} and \ref{iii:omegan-polynomial}. If $\xi_{e_n}(\zeta)=0$, then ${\omega(\zeta)^n=g_n(\zeta)\in \bar{\bF}[\theta]}$. Conversely, if $\omega(\zeta)^n\in \bar{\bF}[\theta]$, we both have $\xi_{e_n}(\zeta)=g_n(\zeta)-\omega(\zeta)^n\in \bar{\bF}[\theta]$ and $|\xi_{e_n}(\zeta)|<1$. Hence, $\xi_{e_n}(\zeta)=0$. 

It remains to prove the equivalence between \ref{iii:omegan-polynomial} and \ref{item:iv-qr-1dividesn}. As $\zeta$ is algebraic of degree $r$ over $\bF$, the relation $\omega^{(1)}=(t-\theta)\omega$ yields:
\[
\omega(\zeta)^{q^r-1}=\prod_{i=0}^{r-1}{\left(\zeta-\theta^{q^i}\right)}.
\]
In particular, if $(q^r-1)|n$, then $\omega(\zeta)^n\in \bar{\bF}[\theta]$. Conversely, consider the polynomial $f(\theta)\in \bar{\bF}[\theta]$ defined by
\[
f(\theta):=\prod_{i=0}^{r-1}{\left(\zeta-\theta^{q^i}\right)^m}=\omega(\zeta)^{n\frac{q^r-1}{q-1}}.
\]
Observe that $f$ has a zero of order $m$ at $\theta=\zeta$. But if $\omega(\zeta)^n$ is a polynomial as well, the order of this zero will be a multiple of $\displaystyle\frac{q^r-1}{q-1}$. This implies $q^r-1|n$, as desired.
\end{proof}

The next proposition is a key step towards determining the torsion of the extension modules, as in Theorem \ref{thm:mot-coh-car}.
\begin{Proposition}\label{prop:gcd-explicit}
Write $n=p^c n_0$ for $c\geq 0$ and $n_0>0$ prime to $p$. Let also $\ell$ be the least common multiple of the integers $r$ such that $q^r-1|n$. Then, the maximal polynomial $\varepsilon_n(t)$ in $\bF[t]$ dividing $e_n$ is
\[
\varepsilon_n(t)=\left(t^{q^\ell}-t \right)^{p^c}.
\]
\end{Proposition}
\begin{proof}[Proof of Proposition \ref{prop:gcd-explicit}]
From Remark \eqref{rem:p-power}, we note that $\varepsilon_{pn}=\varepsilon_n^p$. Accordingly, we can assume without loss of generality that $n$ is prime to $p$. 

Let $r\geq 1$. By definition of $\varepsilon_n(t)$, we have that an element $\zeta\in \bF_{q^r}$ is a zero of $\varepsilon_n(t)$ if and only if it is a zero of $e_n(t)$. From the equivalence between \ref{i:zero-of-e} and \ref{item:iv-qr-1dividesn} in Lemma \ref{lem:equivalence-of-vanishing}, this is the case if and only if $q^r-1$ divides $n$. As $\prod_{\zeta\in \bF_{q^r}} (t-\zeta)=t^{q^r}-t$, and 
$\operatorname{lcm}\left\{t^{q^r}-t|r:~q^r-1|n\right\}=t^{q^{\ell}}-t$, we find $t^{q^\ell}-t|\varepsilon_n(t)$. To claim equality, it suffices to prove that the roots of $e_n(t)$ in $\bar{\bF}$ are simple. 

Suppose that $e_n$ has a double root $\zeta\in \bar{\bF}$. Then, $e(\zeta)=0$ and $(\partial_t e_n)(\zeta)=0$ as well, where $\partial_t$ denotes the partial derivative with respect to $t$. Since $e_n=(t-\theta)^n\xi_{e_n}-\xi_{e_n}^{(1)}$, we have 
\[
\partial_t e_n=(t-\theta)^n(\partial_t \xi_{e_n}) + n(t-\theta)^{n-1}\xi_{e_n}-(\partial_t \xi_{e_n})^{(1)}
\]
so that $0=(\partial_t e_n)(\zeta)=(\zeta-\theta)^n(\partial_t\xi_{e_n})(\zeta)-(\partial_t \xi_{e_n})^{(1)}(\zeta)$ by \ref{ii:zero-of-rho} of Lemma \ref{lem:equivalence-of-vanishing}. Yet, as $\|\xi_{e_n}\|<1$, the same argument as in the proof of \ref{i:zero-of-e} $\Longleftrightarrow$ \ref{ii:zero-of-rho} in Lemma \ref{lem:equivalence-of-vanishing} applies to show that $(\partial_t\xi_{e_n})(\zeta)=0$. We obtain:
\[
(\partial_t \omega^n)(\zeta)=(\partial_t g_n)(\zeta)\in \bar{\bF}[\theta].
\]
By \ref{iii:omegan-polynomial} of Lemma \ref{lem:equivalence-of-vanishing}, we have $\omega(\zeta)^n\in\bar{\bF}[\theta]$ as well. Therefore, $\left(\frac{\partial_t \omega^n}{\omega^n}\right)(\zeta)\in \bar{\bF}(\theta)$. But as $n$ is prime to $p$,
\[
\left(\frac{\partial_t\omega^n}{\omega^n}\right)(\zeta)=n\left(\frac{\partial_t \omega}{\omega}\right)(\zeta)=n\sum_{k=0}^{\infty}{\frac{\theta^{q^k}}{\zeta-\theta^{q^k}}}
\]
is not a rational function over $\bar{\bF}$. This is a contradiction. 
\end{proof}

\subsection{Determination of the class module}
As a first application of the functions $\xi$, we record:
\begin{Proposition}\label{prop:class-module-n-pos}
For $n\geq 0$, the class module $\operatorname{Cl}(\underline{A}(n))$ is trivial.
\end{Proposition}
\begin{proof}
From Corollary \ref{cor:formula-for-carlitz}, it suffices to show that any $f\in \KI\langle t \rangle$ can be written as a sum $f=p+h$ where $p\in (t-\theta)^{-n}\bF[\theta,t]$, and $h\in \KI\langle t \rangle$ is in the image of $\Delta$.

However, for every $h\in \fm_\infty \langle t \rangle$, the series
\begin{equation}\label{eq:s_h}
\xi_{h(t-\theta)^n}=h+\sum_{k=1}^{\infty}{\frac{h^{(k)}}{(t-\theta)^n(t-\theta^q)^n \cdots (t-\theta^{q^{k-1}})^n}}
\end{equation}
converges in $\KI\langle t \rangle$ to an element satisfying $\Delta \xi_{h(t-\theta)^n}=h$. Hence, we can choose the unique elements $p\in \bF[\theta,t]$ and $h\in \fm_\infty \langle t \rangle$ according to the decomposition 
\[
\KI\langle t \rangle = \bF[\theta,t]\oplus  \fm_\infty \langle t \rangle,
\]
and obtain the desired decomposition for $f$.
\end{proof}

\subsection{Polylog classes}
In this subsection, we prove the remaining part of Theorems \ref{thm:polylogclass} and \ref{thm:mot-coh-car} that are concerned with positive twists. We start by constructing a generating set for the integral $t$-motivic cohomology, namely the \emph{polylog classes}. Recall that by Proposition \ref{prop:extensions} we have an $\bF[t]$-linear morphism
\begin{equation}\label{eq:iota-for-context}
\iota:(t-\theta)^{-n}\bF[\theta,t]\longrightarrow \Ext^{1,\operatorname{reg}}_A(\mathbbm{1},\underline{A}(n))
\end{equation}
whose kernel is $\{\Delta p(t)~|~p(t)\in \bF[\theta,t]\}$. We slightly generalize the definition of \emph{polylog classes} given in the introduction:
\begin{Definition}\label{def:polylog}
Let $\bF[\theta,t]_{\deg_\theta<\frac{nq}{q-1}}$ be the set of polynomials whose degree with respect to $\theta$ is less than $\frac{nq}{q-1}$. For $e\in \bF[\theta,t]_{\deg_\theta<\frac{nq}{q-1}}$, we call  \emph{the polylog class of $e$}, and denote by $[L_n(e)]$, the class of the extension obtained as the image of $(t-\theta)^{-n}e$ through $\iota$. 
\end{Definition}

In virtue of Propositions \ref{prop:regulation} and \ref{prop:description-rB}, the extension $[L_n(e)]$ is  regulated and has analytic reduction at $\infty$; indeed, the assumption on $e$ ensures the convergence of the series $\xi_e\in \KI\langle t \rangle$ defined in \eqref{eq:xi}, which satisfies: $\displaystyle \Delta \xi_e=\frac{e}{(t-\theta)^n}$.

\subsubsection*{Proof of Theorem \ref{thm:polylogclass}}
We next show that polylog classes generate the integral $t$-motivic cohomology of~$\underline{A}(n)$, proving the first half of Theorem \ref{thm:polylogclass}.\\

The map $\xi:\bF[\theta,t]_{\deg_\theta<\frac{nq}{q-1}}\to \KI\langle t \rangle$, which assigns $\xi_e$ to $e$, is $\bF[t]$-linear and gives rise -- in conjunction with Corollary \ref{cor:formula-for-carlitz} -- to a commutative diagram:
\begin{equation}\label{eq:diagram-for-context}
\begin{tikzcd}
\bF[\theta,t]_{\deg_\theta<\frac{nq}{q-1}} \arrow[r,"\xi"]\arrow[d,"\times (t-\theta)^{-n}"'] & \displaystyle\frac{\{\xi\in \KI\langle t \rangle ~|~\Delta \xi \in (t-\theta)^{-n}\bF[\theta,t]\}}{\bF[\theta,t]+\mathbf{1}_{q-1|n}\bF[t]\omega(t)^n} \arrow[d,"\wr"] \\
(t-\theta)^{-n}\bF[\theta,t]_{\deg_\theta<\frac{nq}{q-1}} \arrow[r,"\iota"] & \Ext^{1,\operatorname{reg}}_{A,\infty}(\mathbbm{1},\underline{A}(n))
\end{tikzcd}
\end{equation}
where $\iota$ was given in \eqref{eq:iota-for-context}.

\begin{Proposition}\label{prop:generators-of-Hmot}
The upper horizontal map in diagram \eqref{eq:diagram-for-context}, when restricted to the submodule $\bF[\theta,t]_{\deg_\theta<n}$, is surjective. In particular,  if the family $(\alpha_1,\ldots,\alpha_n)$ is a basis of $\bF[\theta]_{\deg_\theta<n}$ over $\bF$, then the family of classes $([L_n(\alpha_1)],\ldots,[L_n(\alpha_n)])$ generates $\Ext^{1,\operatorname{reg}}_{A,\infty}(\bbone,\underline{A}(n))$ over $\bF[t]$.
\end{Proposition}
\begin{proof}
To prove the surjectivity, we consider $\xi\in \KI\langle t \rangle$ and $e\in \bF[\theta,t]$ such that $(t-\theta)^n\xi-\xi^{(1)}=e$. Because $\KI\cs{t}$ decomposes as $\bF[\theta,t]\oplus \fm_{\infty}\cs{t}$, up to subtracting from $\xi$ an element of $\bF[\theta,t]$ one can assume that the coefficients of $\xi$ are all in $\fm_{\infty}$, that is $\|\xi\|<1$. The relation among $\xi$ and $e$ then implies $\|e\|<q^n$; i.e. $\deg_\theta(e)<n$. Hence $\xi$ is equivalent to $\xi_e$ modulo an element of $\underline{A}(n)^+_{B}=\mathbf{1}_{q-1|n}\bF[t]\omega(t)^n$. 
\end{proof}

\begin{Remark}
Let $e\in \bF[\theta,t]_{\deg_\theta<\frac{nq}{q-1}}$. As a corollary of Proposition \ref{prop:generators-of-Hmot}, we get that any $\xi_e$ can be written as an $\bF[t]$-linear combination of $\xi_{e'}$, $\omega(t)^n$ and a polynomial in $\bF[\theta,t]$ for some $e'$ in $\bF[\theta,t]_{\deg_\theta<n}$. This could also be deduced from the relation \eqref{eq:polynomial-xi}.
\end{Remark}

The second half of Theorem \ref{thm:polylogclass} is achieved by the next proposition.

\begin{Proposition}\label{prop:polylogclass-linear-dep}
Let $\alpha_1,\ldots,\alpha_s$ be elements in $\bF[\theta]$ of degree $<\frac{qn}{q-1}$. For all tuples $(a_1(t),\ldots, a_s(t))\in \bF[t]^s$ the following are equivalent:
\begin{enumerate}[label=$(\arabic*)$]
    \item \label{item:relation-extension} $a_1(t)\cdot [\underline{L}(\alpha_1)]+\ldots+a_s(t)\cdot[\underline{L}(\alpha_s)]=0$ in $\Ext^{1,\operatorname{reg}}_{A,\infty}(\bbone,\underline{A}(n))$,
    \item\label{item:relation-xi-function} $a_1(t)\xi_{\alpha_1}(t)+\ldots+a_s(t)\xi_{\alpha_s}(t)\in \bF[\theta,t]+\bF[t]\cdot \omega(t)^n$,
    \item\label{item:relation-polylog} $a_1(\theta)\Li_n(\alpha_1)+\ldots+a_s(\theta)\Li_n(\alpha_s)\in \bF[\theta]\cdot \tilde{\pi}^n$ in $\KI^s$.
\end{enumerate}
\end{Proposition}

The proof of Proposition \ref{prop:polylogclass-linear-dep} uses the Anderson-Brownawell-Papanikolas (ABP) criterion in a crucial way. We recall its statement in the form we use (\emph{cf.} \cite[Thm. 3.1.1 + Prop. 3.1.3]{abp}).

\begin{Theorem*}[ABP criterion]
Fix a matrix $\Phi\in \Mat_\ell(\overline{K}[t])$ such that $\det \Phi$ is a polynomial in $t$ vanishing (if at all) only at $t=\theta$. Fix a (column) vector $\psi\in \Mat_\ell(\CI\langle t \rangle)$ satisfying the functional equation $\psi^{(-1)}=\Phi \psi$. Then, coefficients of $\psi$ are in $\CI\langle\langle t \rangle\rangle$ (\emph{cf.} \eqref{eq:entire-series}) and we can evaluate $\psi$ at $t=\theta$, thus obtaining a column vector $\psi(\theta) \in \Mat_{\ell\times 1}(\KI)$. For every (row) vector $\rho\in \Mat_{1\times \ell}(\overline{K})$ such that $\rho\psi(\theta) = 0$ there exists a (row) vector $P\in \Mat_{1\times \ell}(\overline{K}[t])$ such that $P(\theta) = \rho$, $P\psi=0$.
\end{Theorem*}

We are going to apply the ABP criterion to the system
\[
\psi^{(-1)}=\Phi \psi,
\]
where $\Phi$ and $\psi$ are the matrices
\begin{equation}\label{eq:ABP-system}
\Phi:=\begin{pmatrix}
(t-\theta)^n & 0 & \cdots & \cdots & 0 \\
-\alpha_1 & 1 & \ddots & & \vdots  \\
\vdots & \vdots & \ddots & \ddots & \vdots \\
-\alpha_s & 0 & \cdots & 1 & 0 \\
0 & 0 & \cdots & 0 & 1 
\end{pmatrix}, \quad \psi^{(-1)}:=\omega(t)^{-n}\begin{pmatrix} 1 \\
-\xi_{\alpha_1} \\
\vdots \\ -\xi_{\alpha_s} \\ \omega(t)^n
\end{pmatrix}.
\end{equation}
From the relation \eqref{eq:xi-and-polylog}, we have $\psi|_{t=\theta}=\tilde{\pi}^{-n}\cdot (1,
\alpha_1- \Li_n(\alpha_1),
\ldots, \alpha_s- \Li_n(\alpha_s),\tilde{\pi}^{n})$.

\begin{proof}[Proof of Proposition \ref{prop:polylogclass-linear-dep}]
Equivalence between \ref{item:relation-extension} and \ref{item:relation-xi-function} follows from the commutativity of diagram \eqref{eq:diagram-for-context}, and we are left to show the equivalence between \ref{item:relation-xi-function} and \ref{item:relation-polylog}.\\
We begin to show \ref{item:relation-xi-function}$\Longrightarrow$\ref{item:relation-polylog}. Assume:
\begin{equation}\label{eq:a-relation-among-xi-and-omega}
a_1(t)\xi_{\alpha_1}(t)+\ldots+a_s(t)\xi_{\alpha_s}(t)\in \bF[\theta,t]+\bF[t]\cdot \omega(t)^n.
\end{equation}
Multiplying the above by $(t-\theta)^n$ and evaluating at $t=\theta$ yields $a_1(\theta)\Li_n(\alpha_1)+\ldots+a_s(\theta)\Li_n(\alpha_s)\in \bF[\theta]\cdot \tilde{\pi}^n$, proving \ref{item:relation-polylog}. \\
We will prove \ref{item:relation-polylog}$\Longrightarrow$\ref{item:relation-xi-function} using induction on $s$. The case $s=0$ is trivial, so we assume $s>0$ and that the proposition holds for the $s-1$ elements $\alpha_1,\ldots, \alpha_{s-1}$. Suppose $(a_1,\ldots, a_s)\in \bF[t]^s$ satisfies
\ref{item:relation-polylog}: that is, we have 
\begin{equation}\label{eq:a-relation-among-Li-and-pi}
 a_1(\theta)\Li_n(\alpha_1)+\ldots+a_s(\theta)\Li_n(\alpha_s)+b(\theta)\tilde{\pi}^n=0
\end{equation}
for some $b(\theta)\in \bF[\theta]$. If $a_s(\theta)=0$ and therefore $a_s(t)=0$, the induction hypothesis implies that $(a_1,\ldots, a_s)$ also satisfies \ref{item:relation-extension}. So we may assume that $a_s(\theta)\ne 0$.

In order to lift the relation \eqref{eq:a-relation-among-Li-and-pi} to functions $\xi$, we rewrite it as:
\begin{align*}
0&= \tilde{\pi}^{-n} \left( 
\sum_i a_i(\theta)\alpha_i - a_1(\theta) (\alpha_1-\Li_n(\alpha_1))-\cdots - a_s(\theta)(\alpha_s-\Li_n(\alpha_s))
+ b(\theta)\tilde{\pi}^n \right) \\
&= \rho \cdot \psi(\theta)
\end{align*}
where $\rho$ denotes the row vector $(\sum_i a_i(\theta)\alpha_i, -a_1(\theta),\ldots, -a_s(\theta), b(\theta))$. We are in the range of application of the ABP criterion to the system \eqref{eq:ABP-system} described above. It implies that there exists a row vector $P(t)=(\tilde{a}_0,-\tilde{a}_1,\ldots,-\tilde{a}_s,\tilde{b})$ with coordinates in $\overline{K}[t]^{s+1}$ such that $P(\theta)=\rho$ and $P\cdot \psi=0$. We get:
\begin{equation}\label{eq:P*psiext}
\tilde{a}_1(t) \xi_{\alpha_1}(t)+\cdots +\tilde{a}_s(t) \xi_{\alpha_s}(t)
+\tilde{b}\omega^n(t)= \frac{1}{(t-\theta)^n}\left(\sum_{i}{\tilde{a}_i\alpha_i}-\tilde{a}_0\right). 
\end{equation}
In particular, this implies $(\tilde{a}_1,\ldots,\tilde{a}_s,\tilde{b})\in \cZ(\underline{\alpha})$, and we deduce that there exists a tuple in $\cZ(\underline{\alpha})$ whose $s$-th coordinate is non-zero.

Since $\cZ(\underline{\alpha})$ is generated by elements with coordinates in $\bF[t]^{s+1}$ (Proposition \ref{prop:invariant-generators}), there is some tuple $(c_1,\ldots, c_s,d)$ in $\bF[t]^{s+1}\cap \cZ(\underline{\alpha})$ with $c_s\ne 0$. By Lemma \ref{lem:equiv-cond-for-tuple}, the tuple $(c_1,\ldots, c_s)$ satisfies \ref{item:relation-xi-function}, and thus also \ref{item:relation-polylog} by the first part of the proof. Therefore, the linear combination
\[ v=c_s\cdot (a_1,\ldots, a_s)-a_s\cdot (c_1,\ldots,c_s)\]
satisfies \ref{item:relation-polylog}. As the last component of $v$ is $0$, the induction hypothesis implies that $v$ also satisfies \ref{item:relation-xi-function}.
We deduce that
\[ c_s\cdot (a_1,\ldots, a_s)=v+a_s\cdot (c_1,\ldots,c_s)\]
satisfies \ref{item:relation-xi-function}, and therefore
$(c_sa_1,\ldots, c_sa_s,d)\in \cZ(\alpha)$ for some $d\in \bF[t]$, by Lemma \ref{lem:equiv-cond-for-tuple}. The implication \ref{item:relation-xi-function} $\Rightarrow$ \ref{item:relation-polylog} yields
\[ c_s(\theta) a_1(\theta)\Li_n(\alpha_1)+\ldots+c_s(\theta)a_s(\theta)\Li_n(\alpha_s)+d(\theta)\tilde{\pi}^n=0. \]
On the other-hand, by multiplying \eqref{eq:a-relation-among-Li-and-pi} by $c_s(\theta)$, we get
\[ c_s(\theta) a_1(\theta)\Li_n(\alpha_1)+\ldots+c_s(\theta)a_s(\theta)\Li_n(\alpha_s)+c_s(\theta)b(\theta)\tilde{\pi}^n=0. \]
Subtracting the two relation gives $d(\theta)=c_s(\theta)b(\theta)$.
This implies
\[
(a_1,\ldots, a_s,b)=\frac{1}{c_s}\cdot (c_sa_1,\ldots, c_sa_s,d)
\in \cZ(\alpha)_{\overline{K}(t)}\cap \bF[t]^{s+1}= \cZ(\alpha)\cap \bF[t]^{s+1}.
\]
To conclude, we apply Lemma \ref{lem:equiv-cond-for-tuple} again to deduce that $(a_1,\ldots, a_s)$ satisfies \ref{item:relation-xi-function}.
\end{proof}

\subsubsection*{Proof of Theorem \ref{thm:mot-coh-car}}
We now turn to the proof of Theorem \ref{thm:mot-coh-car}, namely:
\begin{Theorem}\label{thm:positive-twists}
For all positive integers $n$, we have isomorphisms of $\bF[t]$-modules:
\[
\Ext^{1,\operatorname{reg}}_{A,\infty}(\mathbbm{1},\underline{A}(n)) \cong \begin{cases} \bF[t]^{n} & \text{if~} q-1\nmid n, \\ \bF[t]^{n-1}\oplus \bF[t]/(\varepsilon_n(t)) & \text{if~} q-1\mid n,  \end{cases}
\]
where, for $n>0$ and $q-1|n$, the polynomial $\varepsilon_n$ is determined as follows. Let $m=\frac{n}{q-1}$, and write $m=p^c m_0$ for $c\geq 0$ and $m_0>0$ prime to $p$. Let also $\ell$ be the least common multiple of the integers $r$ such that $q^r-1|n$. Then, we have
\[
\varepsilon_n(t)=\left(t^{q^\ell}-t \right)^{p^c}.
\]
\end{Theorem}

\begin{proof}
Let $\varphi:\bF[\theta,t]_{\deg_\theta<n}\to \Ext^{1,\operatorname{reg}}_{A,\infty}(\bbone,\underline{A}(n))$ be the $\bF[t]$-linear map given by either the upper or lower composition in diagram \eqref{eq:diagram-for-context}. By Proposition \ref{prop:generators-of-Hmot}, $\varphi$ is surjective. \\
If $q-1\nmid n$, Corollary \ref{cor:formula-for-carlitz} states that $\Ext^{1,\operatorname{reg}}_{A,\infty}(\bbone,\underline{A}(n))$ has rank $n$, hence the kernel of $\varphi$ is trivial, and therefore $\varphi$ is an isomorphism, which implies the result. \\
If $q-1|n$, Corollary \ref{cor:formula-for-carlitz} states that $\Ext^{1,\operatorname{reg}}_{A,\infty}(\bbone,\underline{A}(n))$ has rank $n-1$, hence the kernel of $\varphi$ is an $\bF[t]$-submodule of rank $1$. Therefore, the image of $\varphi$, namely $\Ext^{1,\operatorname{reg}}_{A,\infty}(\bbone,\underline{A}(n))$, has the form $\bF[t]^{n-1}\oplus \bF[t]/(d(t))$ for some monic polynomial $d(t)\in \bF[t]$. The relation among a generator of $\ker \varphi$ and $d(t)$ is given as follows:
\begin{Fact}
Let $e\in \bF[\theta,t]_{\deg_\theta<n}$ be a generator of $\ker \varphi$.
Then the polynomial ${d(t)\in\bF[t]}$ is the maximal polynomial in $\bF[t]$ which divides $e$.
\end{Fact}

We claim that the polynomial $e_n(t)$ constructed in \eqref{eq:construction-e} is a generator of $\ker \varphi$; this would finish the proof by Proposition \ref{prop:gcd-explicit}.\\
From Lemma \ref{lem:en-and-gn}, and as $(q-1)|n$, $e_n(t)\in \bF[\theta,t]$. In particular, we deduce from Proposition \ref{prop:Z-alpha} that any $\overline{K}[t]$-linear relation of the form
\[
z_1(t)\xi_{\alpha_1}(t)+\ldots + z_s(t)\xi_{\alpha_s}(t)+y(t)\omega(t)^n\in \overline{K}(t),
\]
where $(\alpha_1,\ldots, \alpha_s)$ is any basis of $\bF[\theta]_{\deg_\theta<n}$ over $\bF$, is a $\overline{K}[t]$-multiple of the relation
\[
\xi_e(t)+\omega(t)^n=g_n(t)\in \bF[\theta,t]
\]
(we used Lemma \ref{lem:en-and-gn} again). Hence, by Corollary \ref{cor:formula-for-carlitz}, $e_n(t)$ is a generator of $\ker \varphi$.
\end{proof}

\section{Algebraic independence of generalized Carlitz polylogarithms}\label{sec:alg-ind}
In this section, we achieve the proof of Theorem \ref{thm:algebraic-independence-of-higher-Carlitz-polylog} and its consequences stated in the introduction. Recall that we denoted by $K$ the field $\bF(\theta)$ and by $\overline{K}$ the algebraic closure of $K$ inside $\CI$. Let $\alpha_1,\ldots,\alpha_s$ be elements in $\overline{K}$ of norm $<q^n$ and let $L$ be the subfield of $\overline{K}$ generated by $\alpha_1,\ldots,\alpha_s$ over $\bF(\theta)$. We denote by $\cO_L$ be the integral closure of $\bF[\theta]$ in $L$. Let also $n$ be a positive integer. 

In order to prove Theorem \ref{thm:algebraic-independence-of-higher-Carlitz-polylog}, we shall study the algebro-difference properties of the series $\xi_{\alpha_i}=\xi_{\alpha_i,n}(t)$ introduced in \eqref{eq:xi}. We shall consider the following assumptions:
\begin{enumerate}[label=$(\Alph*)$]
\item\label{assumption:pnotdividen} $n$ is prime to the characteristic $p$ of $\bF$,
\item\label{assumption:linearly-independent} no $\bF[t]$-linear relation among $(\omega^n,\xi_{\alpha_1},\ldots,\xi_{\alpha_s})$ belongs to the subring $\cO_L[t]$ of $\CI(\!( t )\!)$.
\end{enumerate}
\begin{enumerate}[label=$(B')$]
\item\label{assumption:linearly-independent-prime} no $\bF[t]$-linear relation among $(\omega^n,\xi_{\alpha_1},\ldots,\xi_{\alpha_s})$ belongs to the subfield $L(t)$ of $\CI(\!( t )\!)$.
\end{enumerate}
\begin{Remark}
Assumptions \ref{assumption:linearly-independent} and \ref{assumption:linearly-independent-prime} are in fact equivalent, by Proposition \ref{prop:linear-relation-among-xi}.
\end{Remark}

We denote by $\partial_t^{(k)}$ the $k$th \emph{hyperdifferential operator}, that is, the unique $t$-adically continuous $\CI$-linear map $\CI(\!(t)\!)\to \CI(\!(t)\!)$ satisfying 
\[
\hdt{k}{t^m}=\binom{m}{k}t^{m-k}.
\]
on monomials. Opposed to the classical derivation, one has $\partial_t^{(a)}\circ \partial_t^{(b)}=\binom{a+b}{a}\partial_t^{(a+b)}$. Accordingly, the Leibniz rule becomes:
\begin{equation}\label{eq:leibniz}
\hdt{k}{fg}=\sum_{\ell=0}^k{\hdt{\ell}{f}\hdt{k-\ell}{g}}.
\end{equation}
We call a family of functions $(f_1,...,f_s)$ of $\CI(\!(t)\!)$ \emph{differentially algebraic independent}\footnote{One also encounters the terminology \emph{hypertranscendental}, see \cite{maurischat}} if the family $(\hdt{j_1}{f_1},...,\hdt{j_s}{f_s}|j_1,...,j_s\geq 0)$ is algebraically independent over $\bF(\theta,t)$.\\

In \cite[Thm.~1.3]{maurischat} the second author showed that $\omega$ is differentially algebraic independent. As a first step towards Theorem \ref{thm:algebraic-independence-of-higher-Carlitz-polylog}, we extend this result further to:
\begin{Theorem}\label{thm:main-theorem}
Assume both \ref{assumption:pnotdividen} and \ref{assumption:linearly-independent}. Then, the family $(\omega^n,\xi_{\alpha_1},\ldots,\xi_{\alpha_s})$ is differentially algebraic independent over $\overline{K}(t)$.
\end{Theorem}

The proof of this result will require to show first that they are differentially linearly independent; this is done "by hand" in Subsection \ref{sec:linear-indep} using the difference equations they satisfy. The algebraic independence will then follow from a careful analysis of the difference Galois group in Subsection \ref{sec:algebraic-indep}. There, we will use freely results from Picard-Vessiot theory, recalled in the appendix \ref{app:picard-vessiot}. In Subsection \ref{subsec:algebraic-indep-polylog}, we apply the criterion of Anderson-Brownall-Papanikolas \cite{abp} in the form used by Papanikolas in \cite{papanikolas} to derive Theorem \ref{thm:algebraic-independence-of-higher-Carlitz-polylog}. In Section \ref{sec:regulator-A(n)}, we explain how to derive Proposition \ref{prop:equivalence-beilinson-for-carlitz} from the reference \cite{gazda2}. \\

In the next subsections, unless stated otherwise, we assume both \ref{assumption:pnotdividen} and \ref{assumption:linearly-independent}.

\subsection{Linear independence}\label{sec:linear-indep}
This subsection is devoted to the first step of the main proof, that is:
\begin{Proposition}\label{prop:linear-independence}
The family $(1,\hdt{k}{\frac{\xi_{\alpha_i}}{\omega^n}}|k\geq 0)$ is linearly independent over the field $L(t)\left( \hdt{j}{\omega^n}\mid j\geq 0\right)$.
\end{Proposition}

The proof of Proposition \ref{prop:linear-independence} will require the knowledge of the main properties of the following difference equation:
\[ x^{(1)}-(t-\theta)^n x= g \]
for various elements $g$. In this respect, we start by studying the structure of the field $L(t)\left( \hdt{j}{\omega^n}\mid j\geq 0\right)$ where $g$ will belong to, and its behavior under difference operators. To preserve integral elements, we recast the definition of the operator $\Delta$ as:
\begin{equation}\label{eq:Delta(f)-def}
\Delta: f\mapsto f^{(1)}-(t-\theta)^nf.
\end{equation}

\subsubsection*{Difference equation and prolongation}
For $k$ a non negative integer, we set\footnote{Observe that some $y_k$ might be zero if $n$ is a multiple of $p$; more precisely, $y_k$ is zero whenever $p$ divides $\frac{n}{k}$. This is not the case here thanks to our assumption \ref{assumption:pnotdividen}.}
\[
y_k:=\omega^n\hdt{k}{\frac{1}{\omega^n}}.
\]
In particular, $y_0=1$. These elements are better understood using \emph{prolongation matrices} as in \cite{maurischat}. Recall that they are defined by assigning to square matrices $M$ of size $s\geq 1$ square matrices $\rho_{[k]}(M)$ of size $s(k+1)$ via the formula
\[
\rho_{[k]}(M)=\begin{pmatrix}M & \hdt{1}{M} & \cdots & \hdt{k}{M} \\ 0 & M & \ddots & \vdots \\ \vdots & \ddots & \ddots & \hdt{1}{M} \\ 0 & \cdots & 0 & M
\end{pmatrix}
\]
where the notation $\hdt{k}{M}$ means applying $\partial_t^{(k)}$ to each entry of $M$. From \eqref{eq:leibniz}, one deduces that the resulting map $\rho_{[k]}:\Mat_{s}(\CI[\![t]\!])\to \Mat_{s(k+1)}(\CI[\![t]\!])$ is a $\CI$-algebra homomorphism. As a consequence, we obtain a difference equation satisfied by the $y_j$, recognizing that
\begin{equation}\label{eq:matrix-for-yk}
\omega^{n}\rho_{[k]}(\omega^{-n})
= \begin{pmatrix}1 & y_1 & \cdots & y_k \\ 0 & 1 & \ddots & \vdots \\ \vdots & \ddots & \ddots & y_1 \\ 0 & \cdots & 0 & 1 
\end{pmatrix}.
\end{equation}
Namely, from the difference equation $(\omega^n)^{(1)}=(t-\theta)^n\omega^n$, we obtain
\begin{equation}\label{eq:relation-yk}
\left(\omega^{n}\rho_{[k]}(\omega^{-n})\right)^{(1)}=\left((t-\theta)^{n}\rho_{[k]}((t-\theta)^{-n})\right)\cdot \left(\omega^{n}\rho_{[k]}(\omega^{-n})\right).
\end{equation}

We introduce an increasing family of subrings of $\CI\ls{t}$:
\[
k\geq 0:\quad R_k:= L(t)[y_1,...,y_k]
\]
and we let $R$ be their union. We let $F_k$ be the quotient field of $R_k$ and $F =\bigcup_{k\geq 0} F_k$ their union (that is, the quotient field of $R$). From \eqref{eq:relation-yk}, the twist $f\mapsto f^{(1)}$ preserves $R_k$, hence $R$, $F_k$ and $F$. \\

\begin{Remark}\label{rem:discription-F-omega-n}
From the identity
\[
\begin{pmatrix}\omega^n & \hdt{1}{\omega^n} & \cdots & \hdt{k}{\omega^n} \\ 0 & \omega^n & \ddots & \vdots \\ \vdots & \ddots & \ddots & \hdt{1}{\omega^n} \\ 0 & \cdots & 0 & \omega^n
\end{pmatrix} =\omega^n\cdot \left( 
\omega^{n}\rho_{[k]}(\omega^{-n})\right)^{-1}
=\omega^n\cdot \begin{pmatrix}1 & y_1 & \cdots & y_k \\ 0 & 1 & \ddots & \vdots \\ \vdots & \ddots & \ddots & y_1 \\ 0 & \cdots & 0 & 1 
\end{pmatrix}^{-1},
\]
valid for all $k\geq 1$, we see that the field $F(\omega^n)$ is the same as $L(t)( \hdt{j}{\omega^n}\mid j\geq 0)$.
\end{Remark}

\subsubsection*{Properties of the ring $R$}
In view of \cite[Thm.~1.3]{maurischat}, the assumption \ref{assumption:pnotdividen} ensures that $y_1$, ..., $y_k$ are algebraically independent over $L(t)$ for all $k\geq 1$. In particular, $R_k$ is a polynomial ring in $k$ variables over $L(t)$. On $R_k$, we consider the  \emph{lexicographic degree}:  
\[
\operatorname{deglex}_k:R_k=L(t)[y_1,\ldots,y_k]\longrightarrow (\bN^k\cup \{-\infty\},<),
\]
where $<$ denotes the lexicographic order and $-\infty$ a smallest element, defined on monomials by $\operatorname{deglex}_k(a\cdot y_1^{r_1}\cdots y_k^{r_k})=(r_k,\ldots,r_1)$ if $a\neq 0$ and $-\infty$ otherwise, then extended to $R_k$ by
\[
\operatorname{deglex}_k\left(f\right):=\max\left\{\operatorname{deglex}_k(a_{r_1,\ldots,r_k}\cdot y_1^{r_1}\cdots y_k^{r_k})\right\},
\]
where $f=\sum_{(r_1,\ldots,r_k)}{a_{r_1,\ldots,r_k}\cdot y_1^{r_1}\cdots y_k^{r_k}}$. In particular, $\deg(f)=-\infty$ if and only if $f=0$. In addition, $\operatorname{deglex}$ is \emph{multiplicative}, namely:
\begin{equation}\label{eq:deglex-multiplicative}
\operatorname{deglex}(fg)=\operatorname{deglex}(f)+\operatorname{deglex}(g)
\end{equation}
where addition is understood coordinate-wise in $\bN^k$. If $f$ is non zero and $\operatorname{deglex}_k(f)=(d_k,\ldots,d_1)$, then we call $a_{d_1,\ldots,d_k}\in L(t)\setminus\{0\}$ the \emph{leading coefficient of $f$}.

\begin{Lemma}\label{lem:leading-coefficient}
Let $f\in R_k\setminus\{0\}$ with leading coefficient $\beta$. Then, 
\[
\operatorname{deglex}_k(\Delta(f))=\operatorname{deglex}_k(f),
\]
in particular $\Delta(f)\neq 0$, and its leading coefficient is $\Delta(\beta)=\beta^{(1)}-(t-\theta)^n\beta$.
\end{Lemma}

\begin{proof}
We deduce from \eqref{eq:relation-yk} (or from a direct computation) that $y_j^{(1)}=y_j+r_j$ for some linear combination $r_j\in R_{j-1}$ of $1$, $y_1,\ldots,y_{j-1}$. In particular, the order of $r_j$ is smaller than that of $y_j$.
Hence, if $f=\beta\cdot y_1^{r_1}\cdots y_k^{r_k}+(\text{terms~of~smaller~degree})$, then
\begin{align*}
f^{(1)}&=(\beta\cdot y_1^{r_1}\cdots y_k^{r_k})^{(1)}+(\text{terms~of~smaller~degree}) \\
&=\beta^{(1)}\cdot y_1^{r_1}\cdots y_k^{r_k}+(\text{terms~of~smaller~degree}).
\end{align*} 
Therefore, it remains to exclude the case where the expression $\beta^{(1)}-(t-\theta)^n\beta$ equals $0$. Yet, any such solution is a $\bF(t)$-multiple of $\omega^n$, hence is different from $\beta\in L(t)\setminus \{0\}$. 
\end{proof}

\begin{Lemma}\label{lem:integral-solution}
If $r\in F$ satisfies $\Delta(r)\in R$, then $r\in R$.
\end{Lemma}

\begin{proof}
Let $k$ be the minimal number such that $r\in F_k$ and thus $\Delta(r)\in R_k$. We write $r=f/g$ with $f,g\in R_{k}$. Because $R_k$ is a unique factorization domain, we can further assume that $f/g$ is written in lowest terms. We also assume that the leading coefficient of $g$ is $1$. By assumption, 
\[
\Delta(r)=\frac{f^{(1)}}{g^{(1)}}-(t-\theta)^n \frac{f}{g}\in R_k.
\]
We shall use the following:
\begin{Fact}\label{fact:UFD}
Let $a_1,b_1,a_2,b_2$ be elements in a unique factorization domain $A$ with $b_i\neq 0$ and $a_2$ coprime to $b_2$. If the relation $a_1/b_1-a_2/b_2\in A$ holds in the quotient field of $A$, then $b_2$ divides $b_1$. 
\end{Fact}
\begin{proof}[Proof of the fact]
Let $\pi$ be an irreducible element of $A$ with associated valuation $v_\pi$, and assume that $\pi$ divides $b_2$. In particular, $v_\pi(a_2)=0$, hence $v_\pi(b_1)\geq v_\pi(b_2)$ from the relation. Being true for all irreducible elements dividing $b_2$, we conclude. 
\end{proof}
Although $f/g$ is reduced, the quotient $f^{(1)}/g^{(1)}$ might not be. Using Fact \ref{fact:UFD} and the assumption on the leading coefficient of $g$, we get that $g$ divides $g^{(1)}$; let $u\in R_k$ be such that $g^{(1)}=ug$. By \eqref{eq:deglex-multiplicative} and as $\operatorname{deglex}(g^{(1)})=\operatorname{deglex}(g)$, we obtain $u\in L(t)$. Because the leading coefficient of $g$ (and also of $g^{(1)}$) is $1$, we get $u=1$, thus $g^{(1)}=g$, i.e.~$g\in \bF(t)$ and then $g=1$. That is $r=f\in R_k$.
\end{proof}

\subsubsection*{Proof of the linear independence}
\begin{Lemma}\label{lem:difference-equation-hd-xi-omega}
The element $\hdt{j}{\frac{\xi_{\alpha_i}}{\omega^n}}$ satisfies the following difference equation:
\[
\left( \hdt{j}{\frac{\xi_{\alpha_i}}{\omega^n}}\right)^{(1)} = 
\hdt{j}{\frac{\xi_{\alpha_i}}{\omega^n}} - \alpha_i \cdot \left(\hdt{j}{\frac{1}{\omega^n}}\right)^{(1)},
\]
\end{Lemma}

\begin{proof}
This follows from the equations $\Delta \xi_{\alpha_i}=-\alpha_i$, $\Delta \omega^n=0$ -- where $\Delta$ is as in \eqref{eq:Delta(f)-def} --  together with the fact that hyperdifferentiation and twists commute.
\end{proof}

\begin{proof}[Proof of Proposition \ref{prop:linear-independence}] 
Assume that for some $k\geq 0$, there is a non trivial linear relation
\begin{equation}\label{eq:linear-relation}
\sum_{j=0}^k \sum_{i=1}^s  c_{i,j}\hdt{j}{\frac{\xi_{\alpha_i}}{\omega^n}} +d=0,
\end{equation}
for coefficients $d$, $c_{i,j}\in F(\omega^n)$, $i\in \{1,...,s\}$, $0\leq j\leq k$, and some $c_{i,k}$ non-zero. Among all such relations, we assume that \eqref{eq:linear-relation} was taken to be minimal; i.e. $k$ minimal and for such $k$ the number $\#\{ i \mid c_{i,k}\ne 0\}$ minimal. After multiplying \eqref{eq:linear-relation} by some non zero element in $F(\omega^n)$, we can further assume that some coefficient $c_{i,k}$ equals $1$. In particular, subtracting \eqref{eq:linear-relation} from its twist, and using Lemma \ref{lem:difference-equation-hd-xi-omega} gives the shorter relation
\begin{eqnarray*} 0&=&\sum_{j=0}^k \sum_{i=1}^s  c_{i,j}^{(1)}\left( \hdt{j}{\frac{\xi_{\alpha_i}}{\omega^n}}\right)^{(1)}  +d^{(1)} - \sum_{j=0}^k\sum_{i=1}^s c_{i,j}\hdt{j}{\frac{\xi_{\alpha_i}}{\omega^n}} +d\\
&=& \sum_{j=0}^k \sum_{i=1}^s 
\left( c_{i,j}^{(1)}-c_{i,j}\right)\cdot \hdt{j}{\frac{\xi_{\alpha_i}}{\omega^n}} + d^{(1)} -d -
\sum_{j=0}^k \sum_{i=1}^s  c_{i,j}^{(1)} \alpha_i \cdot \left(\hdt{j}{\frac{1}{\omega^n}}\right)^{(1)}.
\end{eqnarray*}
Since, by assumption on minimality there is no shorter relation, all the coefficients in the relation have to vanish. Hence, we deduce $c_{i,j}=c_{i,j}^{(1)}$ and then $c_{i,j}\in \bF(t)$. The vanishing of the remaining term yields a difference equation satisfied by $d\in F(\omega^n)$:
\begin{equation}\label{eq:dtwist-d}
d^{(1)}-d=(\omega^{-n})^{(1)}\cdot  \sum_{j=0}^{k} e_j \left(  \omega^n\hdt{j}{\frac{1}{\omega^n}}\right)^{(1)}
=(\omega^{-n})^{(1)}\cdot  \sum_{j=0}^{k} e_j y_j^{(1)},
\end{equation}
where we set $e_j:=\sum_{i} c_{i,j}\alpha_i\in \cO_L[t]\otimes_{\bF[t]}\bF(t)$. \\

We claim that \eqref{eq:dtwist-d} has no solution in $F(\omega^n)$, adducing the contradiction. As $\omega$ is transcendental over $F$, $F(\omega^n)$ is isomorphic to the field of rational fractions over $F$. In addition, the difference operator $x\mapsto x^{(1)}-x$ on $F\ls{\omega^n}$ preserves the powers of $\omega^n$. Hence, after embedding $F(\omega^n)$ in $F(\!(\omega^n)\!)$, one recognizes that $d\in \omega^{-n}\cdot F+\bF(t)$. Subtracting an element of $\bF(t)$ to $d$, we may assume that \eqref{eq:dtwist-d} has a solution $d\in \omega^{-n}\cdot F$. 

This hints that \eqref{eq:dtwist-d} is better understood after renormalizing $\tilde{d}:=\omega^n d\in F$. Multiplying \eqref{eq:dtwist-d} by $(\omega^n)^{(1)}=(t-\theta)^n\omega^n$ yields
\[
\Delta(\tilde{d})= \tilde{d}^{(1)}-(t-\theta)^n\tilde{d} =  \sum_{j=0}^{k} e_j y_j^{(1)} \in e_k\cdot y_k+R_{k-1}.
\]
In particular, Lemma \ref{lem:integral-solution} implies $\tilde{d}\in R$.
Furthermore by Lemma \ref{lem:leading-coefficient}, the leading coefficient $\beta\in L(t)\setminus\{0\}$ of $\tilde{d}$ satisfies
\begin{equation}\label{eq:difference-ek}
\beta^{(1)}-(t-\theta)^n \beta = e_k\ne 0.
\end{equation}
However, because of \ref{assumption:linearly-independent}, this equation does not have a solution in $L(t)$; indeed, $f:=\beta-\xi_{e_k}$ satisfies $f^{(1)}=(t-\theta)^n f$, hence is a $\bF[t]$-multiple of $\omega^n$. Yet, this is impossible by assumption \ref{assumption:linearly-independent}.
\end{proof}

\subsection{Algebraic independence}\label{sec:algebraic-indep}
In this paragraph, we complete the proof of Theorem \ref{thm:main-theorem}. As we already know from \cite[Thm.~1.3]{maurischat} that the family $(\hdt{j}{\omega^n}\mid j\geq 0)$ is algebraically independent, the statement of the theorem is equivalent to

\begin{Proposition}\label{prop:algebraic-independence-xi-and-hyper}
For all $k\geq 0$, the family
$(\omega^n,\hdt{j}{\xi_{\alpha_i}}\mid i\in\{1,\ldots,s\},j\in\{0,\ldots,k \})$ is algebraically independent over $F$.
\end{Proposition}

The proof of this proposition consists in an application of difference Galois theory using the linear independence result proved in the previous paragraph. The strategy of the proof is quite similar to that of Papanikolas in \cite{papanikolas}.
The necessary background on difference Galois theory is given in appendix \ref{app:picard-vessiot}.

\begin{Remark}\label{rem:generators-of-cR}
As we considered the elements $\hdt{j}{\frac{\xi_{\alpha_i}}{\omega^n}}$ in the previous paragraph we shall use them here as well, instead of $\hdt{j}{\xi_{\alpha_i}}$. This does not affect the argument, since the field extension of $F$ generated be $\omega^n$ and $\hdt{j}{\frac{\xi_{\alpha_i}}{\omega^n}}$ is the same as the one generated by the family $(\omega^n,\hdt{j}{\xi_{\alpha_i}}\mid i\in\{1,\ldots,s\},j\in\{0,\ldots,k \})$.
\end{Remark}

\subsubsection*{Some notations for matrices}
We shall introduce shorter notations for some matrices occurring next. Let $S$ denote an arbitrary $\bF(t)$-algebra. 
\begin{enumerate}[label=$-$]
\item An underlined letter $\underline{b}$ will stand for an $s$-tuple $(b_1,\ldots,b_s)$ of elements in $S$. 
\item For $a,c\in S$ and $\underline{b}$ an $s$-tuple in $S^s$, we will use the following for short:
\[
\begin{pmatrix} a & \underline{b} \\ & c \end{pmatrix}:=\begin{pmatrix}
a & b_1 & b_2 & \cdots & b_s \\ 0 & c & 0 & \cdots & 0 \\ 0 & 0 & c & & \vdots \\ \vdots & \vdots & & \ddots & 0 \\ 0 & 0 & \cdots & 0 & c
\end{pmatrix}
\in \Mat_{s+1}(S).
\]
We consider the matrices:
\[ \Xi := \begin{pmatrix}1 & \omega^{-n}\underline{\xi} \\  & \omega^{-n} \end{pmatrix} \in \GL_{s+1}(\CI\langle t \rangle),\quad \Theta = \begin{pmatrix}1 & -(t-\theta)^{-n}\underline{\alpha} \\  & (t-\theta)^{-n} \end{pmatrix}\in \GL_{s+1}(\bF(\theta,t)),  \]
where $\underline{\xi}=(\xi_{\alpha_1},\ldots, \xi_{\alpha_s})$ and $\underline{\alpha}=(\alpha_1,\ldots,\alpha_s)$. 
\item We extend this notation to block matrices in the obvious way, e.g. for prolongation matrices:
\end{enumerate}
\begin{equation}
\rho_{[k]}(\Xi)=
\begin{pmatrix}
\begin{matrix} 1\hphantom{1} &  \omega^{-n}\underline{\xi} \\ &  \omega^{-n}  \end{matrix} & \vline & \begin{matrix} 0&  \hdt{1}{\omega^{-n}\underline{\xi}} \\ &  \omega^{-n}y_1\end{matrix} & \vline &  \cdots & \vline & \begin{matrix} 0 &  \hdt{k}{\omega^{-n}\underline{\xi}} \\ & \omega^{-n}y_k \end{matrix} \\
\hline
 & \vline & \begin{matrix} 1\hphantom{1m}  & \omega^{-n}\underline{\xi} \hphantom{11} \\ &  \omega^{-n}\hphantom{11}   \end{matrix} & \vline & \ddots & \vline & \vdots \\
\hline
 & \vline & & \vline & \ddots & \vline & \begin{matrix} 0 &  \hdt{1}{\omega^{-n}\underline{\xi}} \\ &  \omega^{-n}y_1 \end{matrix} \\
\hline
 & \vline & & \vline & & \vline &\begin{matrix} 1\hphantom{1m}  & \omega^{-n}\underline{\xi} \hphantom{11} \\ &  \omega^{-n}\hphantom{11}
 \end{matrix}
\end{pmatrix}\in \GL_{(s+1)(k+1)}(\CI\langle t \rangle). \nonumber
\end{equation}
where empty blocks are zero by convention. Observe that $\Xi$ is a fundamental solution matrix of the difference equation~$X^{(1)} = \Theta \cdot X$. Consequently, $\rho_{[k]}(\Xi)$ is a fundamental solution matrix for
\begin{equation}\label{eq:fundamental-solution-rhok}
X^{(1)}=\rho_{[k]}(\Theta)X.
\end{equation}

\subsubsection*{The difference Galois group}
A \emph{Picard-Vessiot ring} $\cR$ for the difference equation \eqref{eq:fundamental-solution-rhok}
 -- considered over the field $F$ -- is given by
\[
\cR = F[\rho_{[k]}(\Xi),\rho_{[k]}(\Xi)^{-1}] =F\left[\omega^n,\omega^{-n},\hdt{j}{\frac{\xi_{\alpha_i}}{\omega^n}} \,\middle|\, i\in\{1,\ldots,s\},j\in\{0,\ldots,k \}\right]
\]
(see Subsection \ref{subsec:Def-PicVess} of the appendix). To study the algebraic relations among the coefficients of $\rho_{[k]}(\Xi)$, we are led to study its \emph{difference Galois group} $\underline{\operatorname{Gal}}(\cR/F)$.

Given $S$ an  $\bF(t)$-algebra, $a\in S^\times, \underline{b},\underline{c},\ldots, \underline{e}\in S^s$, we write $\gamma_k(a,\underline{b},\underline{c},\ldots,\underline{e})$ for the diagonal block matrix:
\[
\gamma_k(a,\underline{b},\underline{c},\ldots,\underline{e}):=
\begin{pmatrix}
\begin{matrix} 1 & \underline{b} \\ & a \end{matrix} & \vline & \begin{matrix} 0 & \underline{c} \\ & 0 \end{matrix} & \vline &  \cdots & \vline & \begin{matrix} 0 & \underline{e} \\ & 0 \end{matrix} \\
\hline
 & \vline & \begin{matrix} 1 & \underline{b} \\ & a \end{matrix} & \vline & \ddots & \vline & \vdots \\
\hline
 & \vline & & \vline & \ddots & \vline & \begin{matrix} 0 & \underline{c} \\ & 0 \end{matrix} \\
\hline
 & \vline & & \vline & & \vline & \begin{matrix} 1 & \underline{b} \\ & a \end{matrix}
\end{pmatrix}\in \Mat_{(s+1)(k+1)}(S).
\]
Above, the letter "$e$" is suitably interpreted as the $k$th letter in an (infinite) alphabet. We further set
\[
\Gamma_k(S)=\left\{\gamma_k(a,\underline{b},\underline{c},...,\underline{e})~|~a,\in S^\times,~\underline{b},\underline{c},...,\underline{e}\in S^s\right\}.
\]

\begin{Lemma}
The functor $\Gamma_k$ defines a Zariski-closed algebraic subgroup of $\GL_{(s+1)(k+1),\bF(t)}$, and is isomorphic to the semi-direct product $\bG_m\rtimes \bG_a^{s(k+1)}$ where $\bG_m$ acts on $\bG_a^{s(k+1)}$ by scalar multiplication.
\end{Lemma}

\begin{proof}
By construction, $\Gamma_k$ defines a Zariski-closed subset of $\GL_{(s+1)(k+1),\bF(t)}$. A direct computation shows that
\[
\gamma_k(a_1,\underline{b}_1,\ldots,\underline{e}_1)\cdot \gamma_k(a_2,\underline{b}_2,\ldots,\underline{e}_2)=\gamma_k(a_1a_2,\underline{b}_2+a_2\underline{b}_1,\ldots, \underline{e}_2+a_2\underline{e}_1)
\]
and
\[ \gamma_k(a,\underline{b},\underline{c},...,\underline{e})^{-1}= \gamma_k(a^{-1},-a^{-1}\underline{b},-a^{-1}\underline{c},...,-a^{-1}\underline{e}). \]
Hence $\Gamma_k$ form a group, and its subgroup of elements $\gamma_k(1,\underline{b},\underline{c},...,\underline{e})$ is a normal subgroup and isomorphic to $\bG_a^{s(k+1)}$.
Further, the subgroup of elements $\gamma_k(a,\underline{0},\ldots,\underline{0})$ is isomorphic to $\bG_m$. Its action on $\bG_a^{s(k+1)}$ stems from the computation:
\[ \gamma(a,\underline{0},\ldots,\underline{0})^{-1}\cdot 
\gamma_k(1,\underline{b},\underline{c},...,\underline{e})
\cdot \gamma(a,\underline{0},\ldots,\underline{0})
= \gamma(1, a\cdot \underline{b},a\cdot\underline{c},...,a\cdot\underline{e}) \]
from which we deduce the isomorphism $\Gamma_k\cong \bG_m\rtimes \bG_a^{s(k+1)}$.
\end{proof}

\begin{Proposition}\label{prop:upper-bound-Galois-group}
The difference Galois group $\underline{\operatorname{Gal}}(\cR/F)$ is contained in~$\Gamma_k$.
\end{Proposition}

\begin{proof}
Part \ref{item:Z-from-Psi} of Proposition \ref{prop:Galois-theory} gives an expression for the universal element in the difference Galois group; this is how this upper bound for the Galois group is found. In this respect, we compute the matrix
$\rho_{[k]}(\Xi)^{-1}\otimes \rho_{[k]}(\Xi)\in \GL_{(s+1)(k+1)}(\cR\otimes_F \cR)$.\footnote{As explained in the appendix, for matrices $M$ and $N$, the notation $M\otimes N$ is shorthand for the matrix
$\left(\sum_{l} M_{il}\otimes N_{lj}\right)_{ij}$ which equals the product $\left( M_{ij}\otimes 1\right)_{ij}\cdot \left( 1\otimes N_{ij}\right)_{ij}$.} 

As it is usual, we extend the hyperdifferential operators $\hdte{j}$ to $\cR\otimes_F \cR$ using the Leibniz rule, i.e.~
\[ \hdt{j}{r_1\otimes r_2} = \sum_{j_1+j_2=j}  \hdt{j_1}{r_1}\otimes \hdt{j_2}{r_2}. \]
With this definition, $\rho_{[k]}$ extends to a ring homomorphism on $\cR\otimes_F \cR$, and also on $\Mat_{(s+1)(k+1)}(\cR\otimes_F \cR)$.

Since
\[
\Xi^{-1} = \begin{pmatrix}1 & -\underline{\xi} \\  & \omega^{n} \end{pmatrix},
\]
and since $\rho_{[k]}$ is a ring homomorphism, we obtain
\begin{eqnarray*}
\rho_{[k]}(\Xi)^{-1}\otimes \rho_{[k]}(\Xi) &=& \rho_{[k]}\left( \Xi^{-1}\otimes \Xi\right) 
= \rho_{[k]}\left( \begin{pmatrix}1 & -\underline{\xi} \\  & \omega^{n} \end{pmatrix} \otimes \begin{pmatrix}1 & \omega^{-n}\underline{\xi} \\  & \omega^{-n} \end{pmatrix} \right) \\
&=& \rho_{[k]} \begin{pmatrix}
1 & 1\otimes \omega^{-n}\underline{\xi} -\underline{\xi}\otimes \omega^{-n} \\ &  \omega^n\otimes \omega^{-n}
\end{pmatrix}.
\end{eqnarray*}
Therefore, all elements in the difference Galois group are of the form
\[ \begin{pmatrix}
\begin{matrix} 1 & \underline{b} \\ & a \end{matrix} & \vline & \begin{matrix}0  & \underline{c} \\ & a_1 \end{matrix} & \vline &  \cdots & \vline & \begin{matrix}0  & \underline{e} \\ &  a_k \end{matrix} \\
\hline
 & \vline & \begin{matrix} 1 & \underline{b} \\ & a \end{matrix} & \vline & \ddots & \vline & \ddots \\
\hline
 & \vline & & \vline & \ddots & \vline & \begin{matrix} 0 & \underline{c} \\ & a_1 \end{matrix} \\
\hline
 & \vline & & \vline & & \vline & \begin{matrix} 1 & \underline{b} \\ & a \end{matrix}
\end{pmatrix}
\]
It remains to conclude that the coefficients $a_j$ are zero. This follows from the computation:
\begin{eqnarray*} 
\hdt{j}{\omega^n\otimes \omega^{-n}} 
&=& \sum_{j_1+j_2=j} \hdt{j_1}{\omega^n}\otimes \hdt{j_2}{\omega^{-n}}
=  \sum_{j_1+j_2=j} \hdt{j_1}{\omega^n}\otimes \omega^{-n}y_{j_2} \\
&=& \sum_{j_1+j_2=j} \hdt{j_1}{\omega^n}y_{j_2}\otimes \omega^{-n} =\sum_{j_1+j_2=j} \hdt{j_1}{\omega^n}\hdt{j_2}{\omega^{-n}}\omega^{n}\otimes \omega^{-n} \\
&=& \hdt{j}{\omega^n\omega^{-n}}\omega^n \otimes \omega^{-n} 
=0.
\end{eqnarray*}
\end{proof}

\begin{Proposition}\label{prop:Galois-group}
The difference Galois group $\underline{\operatorname{Gal}}(\cR/F)$ equals $\Gamma_k$.
\end{Proposition}

\begin{proof}
First, we show that the image of $\underline{\Gal}(\cR/F)$ under the quotient map $\Gamma_k\to \bG_m, \gamma(a,\underline{b},\ldots, \underline{e})\mapsto a$ is $\bG_m$.

By Galois correspondence this image is isomorphic to the difference Galois group $\underline{\Gal}(F[\omega^n,\omega^{-n}]/F)$.
As $\omega^n$ is transcendental over $F$, the Galois group $\underline{\Gal}(F[\omega^n,\omega^{-n}]/F)$ has dimension $1$. As such, it equals $\bG_m$.

We claim that $\underline{\Gal}(\cR/F)$ also contains the full subgroup $\bG_a^{s(k+1)}\subseteq \Gamma_k$.
Assume that this is not the case, hence $V:=\underline{\Gal}(\cR/F)\cap \bG_a^{s(k+1)}$ is a strict subgroup. Since, $V$ is stable under the $\bG_m$-action,  \cite[Prop.~6.2.3]{papanikolas} shows that $V$ is a linear subspace, i.e. defined by linear polynomials equations. This means that there exist coefficients $c_{ij}\in \bF(t)$, $i\in \{1,\ldots ,s\}$, $j\in \{1,\ldots,k\}$, not all of them zero, such that for all 
$\gamma(1,\underline{b_1},\underline{b_2},\ldots, \underline{b_k})\in \underline{\Gal}(\cR/F)(S)$, we have:
\[  \sum_{i,j} c_{ij} (b_j)_i =0. \]
In particular, for the universal element, we obtain
\begin{eqnarray*} 
0&=& \sum_{i,j} c_{ij}\cdot \left( 1\otimes \hdt{j}{\omega^{-n}\xi_{\alpha_i}} -\hdt{j}{\xi_{\alpha_i}\otimes \omega^{-n}}\right) \\
&=& \sum_{i,j} c_{ij}\otimes \hdt{j}{\omega^{-n}\xi_{\alpha_i}}
- \sum_{i,j} c_{ij} \hdt{j}{\xi_{\alpha_i}\otimes \omega^{-n}} \in \cR\otimes_F \cR.
\end{eqnarray*}
As the sum $\sum_{i,j} c_{ij} \hdt{j}{\xi_{\alpha_i}\otimes \omega^{-n}}$ is in $\cR\otimes_F F(\omega^n)$, we obtain a linear relation for the family 
$(1\otimes 1,1\otimes \hdt{j}{\frac{\xi_{\alpha_i}}{\omega^n}}|i\in \{1,\ldots ,s\},j\in \{1,\ldots,k\})$ in $\cR\otimes_F \cR$ with coefficients in $\cR\otimes_F F[\omega^n,\omega^{-n}]$. Yet, by Proposition \ref{prop:linear-independence}, the family  $(1,\hdt{j}{\frac{\xi_{\alpha_i}}{\omega^n}}|i\in \{1,\ldots ,s\},j\in \{1,\ldots,k\})$ is linearly independent over $F(\omega^n)$ in $\CI\cs{t}$, and hence linearly independent over $\cR\otimes_F F(\omega^n)$ in $\cR\otimes_F \CI\cs{t}$. A contradiction.
\end{proof}

\begin{proof}[Proof of Proposition \ref{prop:algebraic-independence-xi-and-hyper}]
From Proposition \ref{prop:Galois-group}, we know that $\underline{\Gal}(\cR/F)=\Gamma_k$, in particular, \[\dim(\underline{\Gal}(\cR/F))=1+s(k+1).\]
By Proposition \ref{prop:Galois-theory}\eqref{item:dimension}, we deduce that
\[ \trdeg_F(\cR)=1+s(k+1). \]
As the ring $\cR$ is generated over $F$ by the $1+s(k+1)$ elements
$\omega^n$, $\hdt{j}{\omega^{-n}\xi_{\alpha_i}}$ ($i\in \{1,\ldots ,s\},j\in \{1,\ldots,k\}$), and the inverse of $\omega^n$, these $1+s(k+1)$ elements have to be algebraically independent over $F$. By Remark \ref{rem:generators-of-cR}, the same holds for the family $(\omega^n, \hdt{j}{\xi_{\alpha_i}}\mid i\in \{1,\ldots ,s\},j\in \{1,\ldots,k\})$. 
\end{proof}

\subsection{Algebraic independence of generalized Carlitz polylogs}\label{subsec:algebraic-indep-polylog}
Recall that the coefficients $\Li_n^{[j]}(\alpha_i)\in \KI$ were defined by \eqref{eq:higherpolylog}. Equivalently, they are given, for all $j\geq 0$, by
\[  \Li_n^{[j]}(\alpha_i) = \hdt{j}{(t-\theta)^n\xi_{\alpha_i}}|_{t=\theta}. \]
Similarly, we introduce coefficients $\pi_n^{[j]}=\hdt{j}{(\omega^{n})^{(1)}}|_{t=\theta}\in \CI$ corresponding to the expansion of $(\omega^n)^{(1)}$ at $t=\theta$:
\[
(\omega^n)^{(1)}=\pi_n^{[0]}+(t-\theta)\pi_n^{[1]}+(t-\theta)^2\tilde{\pi}_n^{[2]}+\ldots 
\]
In particular, $\pi_n^{[0]}=\tilde{\pi}^n$ is the $n$th power of Carlitz' period. \\

The following result supersedes what we were aiming to prove in Theorem~\ref{thm:algebraic-independence-of-higher-Carlitz-polylog}:
\begin{Corollary}\label{cor:algebraic-independence-polylog}
Under assumptions \ref{assumption:pnotdividen} and \ref{assumption:linearly-independent}, the family
\[
(\pi_n^{[j]},\Li_n^{[j]}(\alpha_i)~|~i\in \{1,...,s\},j\geq 0)
\]
of $\CI$, is algebraically independent over $\bF(\theta)$.
\end{Corollary}

In the remaining of this section, we explain how to derive Corollary \ref{cor:algebraic-independence-polylog} from the ABP criterion and Theorem \ref{thm:main-theorem}. \\

Although the ABP criterion has been recalled above, we present below the form obtained by Papanikolas in \cite[Thm. 5.2.2 and proof]{papanikolas}. Let $K$ be the field $\bF(\theta)$ and let $\overline{K}$ denote its algebraic closure in $\CI$ (we assume implicitly that $L$ is contained in $\overline{K}$). Let $\ell$ be a positive integer. Fix a matrix $\Phi(t)\in \Mat_\ell(\overline{K}[t])$ whose determinant is of the form $c(t-\theta)^s$ for some $s\geq 1$ and $c\in \overline{K}^{\times}$. Let also $\Psi(t)$ be a matrix in $\GL_\ell(\CI\langle t \rangle)$ such that 
\[
\Psi(t)^{(-1)}=\Phi(t)\Psi(t).
\]
From \emph{loc.\,cit.}, $\Psi(t)\in\Mat_\ell(\CI\langle\!\langle t \rangle\!\rangle)$ and one can evaluate entries of $\Psi(t)$ at $t=\theta$. 
\begin{Theorem*}[Papanikolas, Thm.~5.2.2 and proof]
The transcendence degree of $\overline{K}(t)(\psi_{ij}(t))$ over $\overline{K}(t)$ equals that of $\overline{K}(\psi_{ij}(\theta))$ over $\overline{K}$, where $\Psi=(\psi_{ij})_{i,j}$. 
\end{Theorem*}
As a consequence, 

\begin{proof}[Proof of Corollary \ref{cor:algebraic-independence-polylog}]
Consider the matrix 
\[ \left( \omega^n \Xi\right)^{(1)}=
\begin{pmatrix}
(t-\theta)^n\omega^n & (t-\theta)^n\xi_{\alpha_1}-\alpha_1 & \cdots & (t-\theta)^n\xi_{\alpha_s}-\alpha_s \\
&&& \\
0 &1 & & 0\\
\vdots & \ddots & \ddots & \\
0 & \cdots & 0 & 1
\end{pmatrix} \]
and its prolongations $\rho_{[k]}\left(( \omega^n \Xi)^{(1)}\right)$ for an arbitrary large integer $k\geq 0$. The entries of $\rho_{[k]}\left(( \omega^n \Xi)^{(1)}\right)(\theta)$ are either $0$, $1$ or $\tilde{\pi}^n$, $\Li_n(\alpha_i)-\alpha_i$, $\pi_n^{[j]}$ and $\Li_n^{[j]}(\alpha_i)$ for $i\in \{1,\ldots,s\}$, $j\in\{1,\ldots , k\}$. However this matrix -- or rather the difference equation it is a solution of --  does not satisfy the prerequisites of the theorem. Instead, we consider:
\[ \Psi:=\,^{\tr}\rho_{[k]}\left(( \omega^n \Xi)^{(1)}\right)^{-1}\]
which is a solution of the difference equation
\[ \Psi^{(-1)}= \Phi\cdot \Psi\]
where
\[ \Phi:=\rho_{[k]}(\begin{pmatrix}
(t-\theta)^n & 0 & \cdots & \cdots & 0 \\
-\alpha_1 & 1 & \ddots & & \vdots \\
\vdots & 0 & \ddots & \ddots & \vdots \\
-\alpha_s & 0 & \cdots & 0 & 1
\end{pmatrix})\in L[t]. \]
Observe that the determinant of $\Phi$ is $(t-\theta)^{n(k+1)}$ and hence we can apply Papanikolas' theorem to this system. Let
\[
T_k:=\trdeg_{\overline{K}(t)}~\overline{K}(t)\left(\hdt{j}{\omega^n},\hdt{j}{\xi_{\alpha_i}}\,|\, i\in\{1,\ldots, s\}, j\in \{0,\ldots, k\}\right).
\]
By Theorem \ref{thm:main-theorem}, $T_k$ is \emph{maximal}, that is $T_k=(s+1)(k+1)$. As the entries of an invertible matrix and those of its inverse generate the same field, we obtain:
\begin{align*}
T_k &=\trdeg_{\overline{K}(t)}~\overline{K}(t)\left(\rho_{[k]}(\omega^n\Xi)^{(1)}\right) \\
&= \trdeg_{\overline{K}(t)}~\overline{K}(t)(\Psi) \\
&= \trdeg_{\overline{K}}~\overline{K}(\Psi(\theta)) \\
&= \trdeg_{\overline{K}}~\overline{K}\left(\rho_{[k]}(\omega^n\Xi)^{(1)}(\theta)\right) \\
&= \trdeg_{\overline{K}}~\overline{K}\left(\pi_n^{[j]},\Li_n^{[j]}(\alpha_i)\,|\, i\in\{1,\ldots, s\}, j\in\{0,\ldots, k\} \right).
\end{align*}
We deduce that the elements $\pi_n^{[j]},\Li_n^{[j]}(\alpha_i)$ for $i\in\{1,\ldots, s\}, j\in\{0,\ldots, k\}$ are algebraically independent over $\overline{K}$, hence over $K$. Since $k\geq 0$ was taken arbitrarily, this concludes. 
\end{proof}

\section{Regulator of $\underline{A}(n)$}\label{sec:regulator-A(n)}
As stated in the introduction, we have:
\begin{Proposition}\label{prop:equivalence-beilinson-for-carlitz}
Let $(\alpha_1,\ldots,\alpha_s)$ be elements of $\bF[\theta]$ of degree $<n$ such that the family of classes $\left([\underline{L}(\alpha_1)],\ldots,[\underline{L}(\alpha_s)]\right)$ is a basis of 
$\Ext^{1,\operatorname{reg}}_{\infty}(\bbone,\underline{A}(n))$. Then, $\mathrm{Reg}(\underline{A}(n))$ is represented by the matrix 
\begin{equation}\label{eq:matrix-polylog}
\begin{pmatrix}
\Li_n^{[n-s]}(\alpha_1) & \Li_n^{[n-s]}(\alpha_2) & \cdots & \Li_n^{[n-s]}(\alpha_s) \\ 
\Li_n^{[n-s+1]}(\alpha_1) & \Li_n^{[n-s+1]}(\alpha_2) & \cdots & \Li_n^{[n-s+1]}(\alpha_s) \\ 
\vdots & \vdots & \ddots & \vdots \\
\Li_n^{[n-1]}(\alpha_1) & \Li_n^{[n-1]}(\alpha_2) & \cdots & \Li_n^{[n-1]}(\alpha_s) 
\end{pmatrix} \in \Mat_s(\KI).
\end{equation}
\end{Proposition}
In this section, we explain how to obtain the above proposition from the reference \cite{gazda2}. There, a regulator attached to a $t$-motive $\underline{M}$ over $\bF(\theta)$ is introduced (Definition 4.3). In the situation $\underline{M}=\underline{A}(n)$, it is the $\KI$-linear map
\[
\mathrm{Reg}(\underline{A}(n)):\Ext^{1,\operatorname{reg}}_{A,\infty}(\bbone,\underline{A}(n))\otimes_A \KI \longrightarrow \Ext^{1,\operatorname{ha}}_{\mathbf{HP}^+,\infty}(\bbone,\cH(n)^+),
\]
where we denoted by $\cH(n)^+$ the \emph{Hodge-Pink structure equipped with its infinite Frobenius} canonically attached to $\underline{A}(n)$ (see \emph{loc.\,cit.} Definition 3.25), induced by the exactness of the \emph{Hodge-Pink realization functor} (\emph{loc.\,cit.} Proposition 3.26).  

From \emph{loc.\,cit.} Theorem 3.28, we have a  square, where horizontal maps are isomorphism, 
\begin{equation}\label{eq:square-regulator}
\begin{tikzcd}[column sep=2.5em,row sep=2.5em]
\displaystyle\frac{\left\{\xi\in \KI\langle t \rangle|\Delta(\xi)\in \bF[\theta,t]\right\}}{\underline{A}(n)^+_B +\bF[\theta,t]} \otimes_A \KI \arrow[r,"\sim"] \arrow[d,"{[\xi]\mapsto [\xi]}"'] & 
\Ext^{1,\operatorname{reg}}_{A,\infty}(\mathbbm{1},\underline{A}(n))\otimes_A \KI \arrow[d,"\mathrm{Reg}(\underline{A}(n)))"]  \\ \displaystyle\frac{(t-\theta)^{-n}\KI[\![t-\theta]\!]}{\underline{A}(n)^+_B\otimes_{A}\KI+\KI[\![t-\theta]\!]} \arrow[r,"\sim"] & \Ext^{1,\operatorname{ha}}_{\mathbf{HP}^+,\infty}(\bbone,\cH(n)^+)
\end{tikzcd}
\end{equation}
where $\Delta:\KI\langle t \rangle\to \KI\langle t \rangle$ denote the difference operator $\xi\mapsto \xi^{(1)}-(t-\theta)^n\xi$. That this diagram commutes -- with left-hand vertical map given by mapping the class of $\xi$ modulo $\bF[\theta,t]$ to the class of $\xi$ modulo $\KI[\![t-\theta]\!]$ -- results from a computation similar to what is done in the proof of Proposition 4.6 in \emph{loc.\,cit}. \\

Let $(\alpha_1,\ldots,\alpha_s)$ be elements of $\bF[\theta]$ of degree $<n$ such that the family of \emph{polylog extension classes} $([\underline{L}(\alpha_1)],\ldots, [\underline{L}(\alpha_s)])$ forms a basis of the space \begin{equation}\label{eq:intriguing-space}
\Ext^{1,\operatorname{reg}}_{A,\infty}(\mathbbm{1},\underline{A}(n))\otimes_A \KI.
\end{equation}
By Theorems \ref{thm:polylogclass} and \ref{thm:mot-coh-car}, we have $s=n$ if $q-1\nmid n$ or $s=n-1$ if $q-1|n$; in addition, the preimage of this basis through the upper horizontal map of the square \eqref{eq:square-regulator} is given by class of the family $(\xi_{\alpha_1},\ldots,\xi_{\alpha_s})$. On the other-hand, since the valuation of $\omega^n$ at $(t-\theta)$ is precisely $-n$, the class of the family $((t-\theta)^{-s},\ldots, (t-\theta)^{-1})$ forms a basis of the target. Hence The matrix $(\Li_n^{[j]}(\alpha_j))_{i,j}$, $i\in \{1,\ldots s\}$ and $j\in \{n-s,...,n-1\}$, represents the left-hand vertical $\KI$-linear morphism. Proposition \ref{prop:equivalence-beilinson-for-carlitz} follows. 

\begin{Corollary}\label{cor:exact-rank}
The rank of $\mathrm{Reg}(\underline{A}(n))$ is precisely $\frac{n}{p^{v_p(n)}}-\mathbf{1}_{q-1|n}$ where $v_p(n)$ denote the $p$-adic valuation of $n$ and $\mathbf{1}_{q-1|n}$ is the indicator function of $q-1|n$. 
\end{Corollary}
\begin{proof}
Let $c=v_p(n)$ and $m=n/p^c$. By \eqref{eq:p-polylog}, the $j$th line of the matrix \eqref{eq:matrix-polylog} vanishes whenever $p^c\nmid j$. Thus, it suffices to show that the matrix $(\Li_{m}^{[j]}(\alpha_i^{1/p^{c}}))_{i,j}$ has maximal rank, where now $i\in \{1,\ldots, n\}$ and $j\in \{k/p^c~|~k\in\{\mathbf{1}_{q-1|n},\ldots, n-1\},~p^c|k\}$. The latter set has cardinality $n/p^c-\mathbf{1}_{q-1|n}$. Yet, as the classes $[\underline{L}(\alpha_1)],\ldots,[\underline{L}(\alpha_s)]$ are linearly independent, no $\bF[t]$-linear relation among $(\omega^n,\xi_{\alpha_1,n},\ldots,\xi_{\alpha_s,n})$ belongs to $\bF[\theta,t]$ (Proposition \ref{prop:generators-of-Hmot}) and hence no $\bF[t]$-linear relation among $(\omega^{m},\xi_{\alpha_1^{1/p^c},m},\ldots,\xi_{\alpha_s^{1/p^c},m})$ belongs to $\bF[\theta^{1/p^c},t]$. In particular, we can apply Corollary \ref{cor:algebraic-independence-polylog} for $\alpha_1^{1/p^c},\ldots,\alpha_n^{1/p^c}$ to show that none of the corresponding minors vanish. 
\end{proof}

\appendix
\section{Picard-Vessiot theory for linear difference equations}\label{app:picard-vessiot}

We present in this section a short handout of difference Galois theory insisting on results that are used in this article. Statements -- in perhaps different form or phrasing -- can be found in \cite[Sect.~2]{maier-annalen}. For proofs, one is referred to \cite{maier-phd}. The book of van der Put and Singer \cite{vdps} is too restrictive in its assumptions to fit our setting. We still recommend it as an introduction to the subject for the interested reader.

\newcommand{\ck}{\mathsf{k}}
\newcommand{\sA}{\mathsf{A}}

\subsection{Setup}
Let $\cF$ be a field, $\tau:\cF\to \cF$ be a field endomorphism, and let
\[
\ck:=\cF^\tau:=\{ f\in \cF\mid \tau(f)=f \}
\]
be the subfield of $\tau$-invariant elements.

Let $r$ be a positive integer. Given an invertible matrix $\Theta\in \GL_r(\cF)$, we consider the linear difference equation
\begin{equation} \label{eq:lin-diff-eq}
    \tau\begin{pmatrix}
    y_1 \\ \vdots \\ y_r
    \end{pmatrix}= \Theta\cdot \begin{pmatrix}
    y_1 \\ \vdots \\ y_r
    \end{pmatrix}.
\end{equation}
As is usual in the theory, the notation on the left-hand side means applying $\tau$ coefficient-wise to vectors, respectively entry-wise for matrices.

We assume that there exists an integral domain $\cL$, containing $\cF$ as a subring, equipped with an extension of the endomorphism $\tau$ to $\cL$ -- which we shall still denote by $\tau$ -- such that the difference equation \eqref{eq:lin-diff-eq} admits a \emph{fundamental solution matrix} $\Psi$; that is, there exists $\Psi\in \GL_r(\cL)$ satisfying
\[   \tau(\Psi) = \Theta \Psi. \]
We assume further that $\cL^\tau=\cF^\tau=\ck$.
\begin{Remark}\label{rmk:representation-matrix}
Observe that any solution $\Psi_0\in \Mat_r(\cL)$ of \eqref{eq:lin-diff-eq} differs from $\Psi$ by a matrix with coefficients in $\ck$; indeed, $C=\Psi^{-1}\Psi_0$ satisfies $\tau(C)=C$. 
\end{Remark}

\begin{Example}
In the core of the paper, we used $\cF=F=\bF(\theta,t)(y_1,y_2,\ldots)$, $\tau:\cF\to \cF$ for the $\bF(t)$-algebra endomorphism mapping $\theta\mapsto \theta^q$, and $\cL$ for the ring $\CI\cs{t}$.
\end{Example}

\subsection{Definitions from Picard-Vessiot theory}\label{subsec:Def-PicVess}
\paragraph{The Picard-Vessiot ring.} In view of Remark \ref{rmk:representation-matrix}, the subring $\cR:=\cF[\Psi,\det(\Psi)^{-1}]$ of $\cL$, generated by the entries of $\Psi$ and the inverse of its determinant, does not depend on the choice of the fundamental solution matrix $\Psi$. It is called the \emph{Picard-Vessiot ring} for the difference equation \eqref{eq:lin-diff-eq} and will be denoted $\cR=\cF[\Psi,\Psi^{-1}]$ for short.

\paragraph{The Picard-Vessiot ideal $I$.} Let $\cF[X,X^{-1}]:= \cF[X_{ij},\det(X)^{-1}\mid 0\leq i,j\leq r]$ be the localization of the polynomial ring over $\cF$ in $r^2$ indeterminates $X_{ij}$, localized at the polynomial $\det(X)$ -- the determinant of the square matrix $X=(X_{ij})_{i,j}$. We equip this localized polynomial ring with an endomorphism extending $\tau$ on $\cF$ by setting
\[ \tau(X)=\Theta \cdot X \]
coefficient-wise. Then, by construction, the Picard-Vessiot ring $\cR$ is the image of the difference homomorphism
\[
\pi: \cF[X,X^{-1}]\longrightarrow \cL, \quad X_{ij}\longmapsto \Psi_{ij}.
\]
This way, $\cR$ becomes isomorphic, as a difference ring, to the quotient $\cF[X,X^{-1}]/I$ where $I:=\ker(\pi)$. We name $I\subseteq \cF[X,X^{-1}]$ the \emph{Picard-Vessiot ideal}.

\paragraph{The difference Galois group.} The \emph{difference Galois group of $\cR/\cF$}, denoted by $\cG=\underline{\Gal}(\cR|\cF)$, is given as the group functor
\begin{align*}
\cG: ~& \CAlg_\ck & \longrightarrow & \quad \CGrp , \\
& ~\sA & \longmapsto & \quad \cG(\sA):= \left\{
\alpha\in \Aut_{\cF\otimes_\ck \sA}(\cR\otimes_\ck \sA) \,\, \middle|\,\, 
(\tau\otimes \id)\circ \alpha =\alpha\circ (\tau\otimes \id) \right\} &
\end{align*}
Any element $\alpha\in \cG(\sA)$ acts on the set of solutions of \eqref{eq:lin-diff-eq} via $\Psi_0\mapsto \alpha(\Psi_0)$ where, once again, $\alpha$ is applied entry-wise. By Remark \ref{rmk:representation-matrix}, there exists a matrix $C(\alpha)\in \Mat_r(\ck)$ such that $\alpha(\Psi)=\Psi C(\alpha)$. If $C(\alpha)=C(\beta)$, then $\alpha(\Psi)=\beta(\Psi)$ and thus $\alpha=\beta$ on $\cR$. Consequently, the map:
\[
C:\cG(\sA)\longrightarrow \GL_r(\sA), \quad \alpha\longmapsto C(\alpha)
\]
is an injective group homomorphism. It is clear that this embedding is natural in the $\ck$-algebra $\sA$, thus interpreting the difference Galois group $\underline{\Gal}(\cR|\cF)$ as subgroup functor of $\GL_{r/\ck}$. 

\paragraph{The Hopf ideal $J$.} A main result in difference Galois theory states that $\cG$ is a closed subgroup of $\GL_{r/\ck}$ for the Zariski topology. As such it is given as the spectrum of a Hopf algebra. More precisely, in the same way as above, we consider the algebra
$\ck[Z,Z^{-1}]:= \ck[Z_{ij},\det(Z)^{-1}\mid 0\leq i,j\leq r]$ as the Hopf algebra associated to $\GL_{r/\ck}$, then 
\[  \cG=\Spec(\ck[Z,Z^{-1}]/J), \]
for some Hopf ideal $J\subseteq \ck[Z,Z^{-1}]$.

\subsection{Results from Picard-Vessiot theory}
The main results in Picard-Vessiot theory are derived from the relation among the Picard-Vessiot ideal $I$ and the Hopf ideal $J$ that we now describe. Because $\Psi$ is invertible, we have an isomorphism of $\cL$-algebras
\[  \iota: \cL\otimes_\cF \cF[X,X^{-1}] \longrightarrow \cL\otimes_\ck \ck[Z,Z^{-1}], \quad 1\otimes_{\cF} X_{ij}\longmapsto \sum_{\ell=1}^r \Psi_{i\ell}\otimes_{\ck} Z_{\ell j}, \]
or rather, in obvious matrix notations, $\id_r\otimes_{\cF} X\mapsto \Psi\otimes_{\ck} Z$. Then, it is a fact that
\begin{equation}\label{eq:relationIJ}
\iota(\cL\otimes_{\cF} I)=\cL\otimes_{\ck} J.
\end{equation}
Equivalently, we have an induced isomorphism:
\[ \bar{\iota}: \cL\otimes_\cF \cR \stackrel{\sim}{\longrightarrow} \cL\otimes_\ck \ck[\cG]  \]
where $\ck[\cG]:=\ck[Z,Z^{-1}]/J$ is the ring of global sections of $\cG$. In geometric terms, this means that the Picard-Vessiot ring is a principal homogeneous space -- or torsor -- over the Galois group scheme $\cG$.

\begin{Remark}
The isomorphism $\iota$ and the induced isomorphism $\bar{\iota}$ are indeed compatible with the difference operator $\tau$, in the way that
\[ (\tau\otimes \id)\circ \iota = \iota \circ (\tau\otimes \tau). \]
\end{Remark}

Let $\bar{Z}_{ij}\in \ck[\cG]$ be the image of $Z_{ij}$ modulo $J$. Then the matrix $\bar{Z}=(\bar{Z}_{ij})_{i,j}$ is the universal element in $\cG(\ck[\cG])$, i.e. for every $\ck$-algebra $\sA$, every element of $\cG(\sA)$ is a specialization of $\bar{Z}$. The relation among $I$ and $J$ (c.f. \eqref{eq:relationIJ}) has the following consequences which were used in the core of the paper.
\begin{Proposition}\label{prop:Galois-theory}
With notation as above.
\begin{enumerate}
    \item \label{item:dimension} The transcendence degree of $\cR$ over $\cF$ equals the dimension of the algebraic group $\underline{\Gal}(\cR/\cF)$;
    \item \label{item:Z-from-Psi} For all $i,j\in\{1,...,r\}$: $\bar{Z}_{ij}=\bar{\iota}\left((\Psi^{-1}\otimes \Psi)_{ij}\right)$;
    \item \label{item:upper-bound} If $\tilde{\cG}\subseteq \GL_{r/\ck}$ is a closed subgroup defined over $\ck$ such that $\Psi\in \tilde{\cG}(\cL)$, then $\cG\subseteq \tilde{\cG}$;
    \item \label{item:linear} Every linear polynomial in $J$ corresponds to a linear polynomial in $I$, and vice versa. More precisely, the elements $Z_{1j}$ are mapped to a linear combinations of the $X_{1j}$ via $\iota^{-1}$. In particular, if $J$ contains a linear polynomial $P$ in the $Z_{1j}$, then 
    $\iota^{-1}(P)$ is a linear polynomial in the $X_{1j}$, and by \eqref{eq:relationIJ} is in $\cL\otimes_{\cF} I$.
\end{enumerate}
\end{Proposition}

\def\cprime{$'$}

\end{document}